\documentclass{amsart}

\usepackage{amsfonts}
\usepackage{cases}
\usepackage{mathrsfs}
\usepackage{bbm}
\usepackage{amssymb}
\usepackage{txfonts}
\usepackage{amscd}
\usepackage{amsfonts,latexsym,amsmath,amsthm,amsxtra,mathdots}
\usepackage[bookmarks,colorlinks]{hyperref}
\usepackage[all,cmtip]{xy}
\RequirePackage{amsmath} \RequirePackage{amssymb}
\usepackage{color}
\usepackage{colordvi}
\usepackage{multicol}
\usepackage{tikz}
\usepackage{hyperref}
\usepackage{graphicx}
\usepackage{amsmath}
\usepackage{amsmath,amscd}
\usepackage[normalem]{ulem}
\usetikzlibrary{matrix,arrows}
\usepackage{textcomp}
\usepackage{mathtools}
\usepackage{yfonts}
\usepackage{enumitem}   

\usepackage{cleveref}

\usepackage[
    backend=biber,
    style=alphabetic,
    maxbibnames=99
]{biblatex}
  
\AtEveryBibitem{\clearlist{language}} 
\AtEveryBibitem{\clearfield{issn}}
\AtEveryBibitem{\clearfield{isbn}}
\AtEveryBibitem{\clearfield{doi}}

\newcommand{\subgroup}{\leq}

    \newcommand{\BK}{{\mathbb {K}}} 
     \newcommand{\BN}{{\mathbb {N}}}
     
    \newcommand{\BQ}{{\mathbb {Q}}} \newcommand{\BR}{{\mathbb {R}}}

     \newcommand{\BZ}{{\mathbb {Z}}}

     \newcommand{\Aut}{{\mathrm{Aut}}}

     \newcommand{\GL}{{\mathrm{GL}}}

\def\-{^{-1}}

\newcommand{\delete}[1]{}

    \newcommand{\res}{\operatorname{res}} 
    \newcommand{\cd}{\operatorname{cd}} 

    \theoremstyle{plain}

\newtheorem{thm}{Theorem}[section]
\newtheorem{defn}[thm]{Definition} 
\newtheorem{defns}[thm]{Definitions}
\newtheorem{ex}[thm]{Example} 
\newtheorem{lem}[thm]{Lemma}
\newtheorem{prop}[thm]{Proposition}
\newtheorem{cor}[thm]{Corollary}
\newtheorem{rem}[thm]{Remark}

\newtheorem{observation}[thm]{Observation}

\newtheorem*{thm*}{Theorem}

\newtheorem*{rem*}{Remark}

\newtheorem*{cor*}{Corollary}

\newtheorem*{namedthm}{\namedthmname}
\newcounter{namedthm}

\renewenvironment{proof}[1][\proofname]{{\bfseries #1. \quad}}{\qed}

\makeatletter
\newenvironment{named}[1]
  {\def\namedthmname{#1}%
   \refstepcounter{namedthm}%
   \namedthm\def\@currentlabel{#1}}
  {\endnamedthm}
\makeatother

\newcommand{\family}[1]{\mathcal{#1}}

\newcommand{\FIN}{\family{F}in}
\newcommand{\CYC}{\family{C}yc}
\newcommand{\FCYC}{\family{FC}yc} 
\newcommand{\VCYC}{\family{VC}yc}

    \numberwithin{equation}{section}

\newcommand{\uueg}{{\underline{\underline E}}G}

\newcommand{\uu}[1]{\uline{\uline{#1}}}
\renewcommand{\u}[1]{\uline{#1}}

\def\Proof{\noindent{\bf Proof.}\quad}
\def\qed{\hfill$\square$\smallskip}

\newenvironment{cond}[1]{%
  \vspace{1.2ex}\begin{enumerate}%
  \renewcommand*{\theenumi}{#1}%
  \renewcommand*{\labelenumi}{#1}%
  \item%
}{%
  \end{enumerate}\vspace{1.2ex}%
}

\begin{document}

\title{Linear Groups, Conjugacy Growth, and Classifying Spaces for Families of Subgroups}

\author{Timm von Puttkamer}
\address{University of Bonn, Mathematical Institute, Endenicher Allee 60, 53115 Bonn, Germany}
\email{tvp@math.uni-bonn.de}

\author{Xiaolei Wu}
\address{Max Planck Institut f\"ur Mathematik, Vivatsgasse 7, 53111 Bonn, Germany }
\email{hsiaolei.wu@mpim-bonn.mpg.de}


\date{August, 2017}


\begin{abstract}
Given a group $G$ and a family of subgroups $\mathcal{F}$, we consider its classifying space $E_{\mathcal F}G$ with respect to $\mathcal{F}$. When $\mathcal F = \VCYC$ is the family of virtually cyclic subgroups, Juan-Pineda and Leary conjectured that a group admits a finite model for this classifying space if and only if it is virtually cyclic. By establishing a connection to conjugacy growth we can show that this conjecture holds for linear groups. We investigate a similar question that was asked by L\"uck--Reich--Rognes--Varisco for the family of cyclic subgroups. Finally, we construct finitely generated groups that exhibit wild inner automorphims but which admit a model for $E_{\VCYC}(G)$ whose 0-skeleton is finite.

\end{abstract}

\maketitle

\section*{Introduction}

Given a group $G$, a family $\mathcal{F}$ of subgroups of $G$ is a set of subgroups of $G$ which is closed under conjugation and taking subgroups. We denote by $ E_{\mathcal{F}}(G)$ a $G$-CW-model for the classifying space for the family $\mathcal{F}$. The space $E_{\mathcal{F}}(G)$ is characterized by the property that the fixed point set $E_{\mathcal{F}}(G)^H$ is contractible for any $H \in \mathcal{F}$ and empty otherwise. Recall that a $G$-CW-complex $X$ is said to be finite if it has finitely many orbits of cells. Similarly, $X$ is said to be of finite type if it has finitely many orbits of cells of dimension $n$ for any $n$. We abbreviate $E_{\mathcal{F}}(G)$ by $\uueg$ for $\mathcal{F} = \VCYC$ the family of virtually cyclic subgroups, $\u EG$ for $\mathcal{F} = \FIN$ the family of finite subgroups and $EG$ for $\mathcal{F}$ the family consisting only of the trivial subgroup. In \cite[Conjecture 1]{JL}, Juan-Pineda and Leary formulated the following conjecture:

\begin{named}{Conjecture A}   \cite [Juan-Pineda and Leary]{JL}\label{ConjA}
Let $G$ be a group admitting a finite model for $\uueg$. Then $G$ is virtually cyclic.
\end{named}

This conjecture has been proven for hyperbolic groups in the same paper and its validity has been extended, for example to elementary amenable groups \cite{KMN}, acylindrically hyperbolic groups, 3-manifold groups, one-relator groups and CAT(0) cube groups \cite{vPW}. Note that a group $G$ admits a model for $\uueg$ with a finite $0$-skeleton if and only if the following holds

\begin{cond}{(BVC)}\label{cond:BVC}
\textit{$G$ has a finite set of virtually cyclic subgroups $\{V_1,V_2,\ldots, V_n\}$ such that every virtually cyclic subgroup of $G$ is conjugate to a subgroup of some $V_i$.}
\end{cond}

Following Groves and Wilson, we shall call this property BVC and the finite set $\{V_1,V_2,\\\ldots, V_n\}$ a \emph{witness} to BVC for $G$. So far almost all the proofs of \ref{ConjA} boil down to verifying whether the group has BVC. In fact, the following conjecture was made in \cite[Conjecture B]{vPW}

\begin{named}{Conjecture B}\label{conj-bvc}
Let $G$ be a finitely presented group which has BVC. Then $G$ is virtually cyclic.
\end{named}

Note that there are many finitely generated torsion-free groups with BVC that are not virtually cyclic, some groups with additional properties are constructed in \Cref{fg-group-BVC}. By establishing a connection between conjugacy growth and BVC, our first theorem verifies \ref{conj-bvc} and hence \ref{ConjA} for linear groups. 

\begin{named}{Theorem I}\label{thmA}[\ref{linear-bVCYC}]
A finitely generated linear group has BVC if and only if it is virtually cyclic.
\end{named}

The theorem also puts very strong restrictions on linear representations of a group with BVC.

\begin{named}{Corollary} [\ref{linear-rep-BVC}]
Let $\varphi \colon G \rightarrow L$ be a surjective homomorphism where $G$ is a finitely generated group with BVC and $L$ is linear. Then $L$ is virtually cyclic. 
\end{named}

Note that, strictly speaking, this does not directly follow from \ref{thmA} if the linear group $G$ is not virtually torsion-free, since the BVC property is not inherited by quotients in general (see \Cref{bvc-no-inheritance-to-quotients}). This was one reason for us to introduce the weaker notion $b\VCYC$, or more generally $b \family F$  for a family of subgroups $\family F$, that is inherited by quotients, see \Cref{basic}.

When $\family{F} = \CYC$, the family of cyclic subgroups, a similar question was asked by  L\"uck--Reich--Rognes--Varisco \cite[Question 4.9]{LRRV}: 

\begin{named}{Question A} \label{quesC}
Does a group $G$ have a model of finite type for $E_{\CYC}G$ if and only if $G$ is finite, cyclic or infinite dihedral?
\end{named}

The case of the infinite dihedral group was missing in the original formulation of the question in \cite{LRRV} but has to be included as we show in \Cref{inf-dih-fintype}. 

Similar to the case of the family of virtually cyclic subgroups, a group $G$ admits a model for $E_{\mathcal{C}yc}G$ with a finite $0$-skeleton if and only if the following holds

\begin{cond}{($b \CYC$)}\label{cond:bCYC}
\textit{ $G$ has a finite set of  cyclic subgroups $\{C_1,C_2,\ldots, C_n\}$ such that every  cyclic subgroup of $G$ is conjugate to a subgroup of some $C_i$.}
\end{cond}

It turns out that once we answer \ref{quesC} positively for virtually cyclic groups, most of the proofs of \ref{ConjA} also apply to \ref{quesC}. In fact, we prove in \Cref{cyclic-EF}  the following

\begin{named}{Theorem II}\label{thm:ecyc-g-conjecture}
For the following classes of groups, \ref{quesC} has a positive answer.
\begin{enumerate}
    \item elementary amenable groups
    \item one-relator groups
    \item acylindrically hyperbolic groups
    \item 3-manifold groups
    \item CAT(0) cube groups and
    \item linear groups
\end{enumerate}
\end{named}

In analogy to \ref{conj-bvc}, we can phrase the following
\begin{named}{Question B}\label{ques:bCYC}
Suppose $G$ is a finitely presented group which has $b \CYC$. Is $G$ necessarily finite, cyclic, or infinite dihedral?
\end{named}

Once more, we obtain a proof of \ref{thm:ecyc-g-conjecture} essentially by answering \ref{ques:bCYC} with the exception of the class of elementary amenable groups. Here, we additionally rely on stronger finiteness conditions that are imposed on the group by having a model of finite type for $E_{\CYC}(G)$.

In the last section of this paper we construct groups which can serve as counterexamples for various reasonable questions regarding the BVC resp. $b\CYC$ property. For example, in \Cref{VC--undistorted-linear} we show that a finitely generated group with BVC whose infinite cyclic subgroups are quasi-isometrically embedded can have at most linear conjugacy growth. Part (c) of the following theorem shows that one cannot dispense with the assumption on quasi-isometrically embedded infinite cyclic subgroups. The main tools for the constructions are HNN extensions and small cancellation theory over relatively hyperbolic groups as developed in \cite{Os} and \cite{HuOs}. 

\begin{named}{Theorem III}
\begin{enumerate}
    \item  There exists a finitely generated torsion-free group $G$ with a finite index subgroup $H$ such that $H$ has $b\CYC$ but $G$ does not. [\ref{finite-index-over-bvc}]
    \item There exists a finitely generated torsion-free group $G = H \rtimes \BZ$ such that $G$ has $b\CYC$ but $H$ does not. [\ref{bvc-fg-extension-by-integers}]
    \item  There exists a finitely generated torsion-free group $G$ with $b\CYC$ that has exponential conjugacy growth. [\ref{bvc-exp-conj-growth}]
\end{enumerate}
\end{named}

Note that Leary and Nucinkis \cite{LN03} give examples of groups $H\leq G$ where $H$ is a finite index subgroup of $G$ such that $H$ has a finite model for $EH$ ($=\u EH$ here) but $G$ has infinitely many conjugacy classes of finite subgroups.

Results of this paper will also appear as part of the first author's thesis.

\textbf{Acknowledgements.} The first author was supported by an IMPRS scholarship of the Max Planck Society. The second author would like to thank the Max Planck Institute for Mathematics at Bonn for its support. We also would like to thank Denis Osin for helpful discussions regarding the proof of \cite[Theorem 7.2]{HuOs}.

\section{Groups Admitting a Finite Model for $\uueg$ and Property $b\mathcal{F}$}\label{basic}
In this section we first review some properties of and results on groups admitting a finite model for $\uueg$. Most of this material is taken directly from \cite[Section 1]{vPW}. Afterwards, we introduce the notion $b\mathcal{F}$ for a family of subgroups $\mathcal{F}$, which, for $\family F = \VCYC$ is slightly weaker but more flexible than the BVC notion. 

We summarize the properties of groups admitting a finite model for $\uueg$ as follows

\begin{prop}
Let $G$ be a group admitting a finite model for $\uueg$, then

\begin{enumerate}
    \item $G$ has BVC.

    \item $G$ admits a finite model for $\underline{E}G$. 

    \item For every finite subgroup of $H \subset G$, the Weyl group $W_GH$ is finitely presented and of type $FP_{\infty}$. Here $W_GH = N_G(H)/H$, where $N_G(H)$ is the normalizer of $H$ in $G$.

    \item $G$ admits a model of finite type for $EG$. In particular, $G$ is finitely presented.

\end{enumerate}

\end{prop}

\begin{rem}
If one replaces finite by finite type in the assumptions of the above proposition, then the conclusions still hold if one also replaces finite by finite type in (b).
\end{rem}

\begin{lem} \cite[Lemma 1.3]{vPW} \label{BVCfinitezeroskeleton}
Let $G$ be a group. There is a model for $\uueg$ with finite $0$-skeleton if and only if $G$ has BVC.
\end{lem}

The following structure theorem about virtually cyclic groups is well known, see for example \cite[Proposition 4]{JL} for a proof.

\begin{lem}  \label{vcstru}
Let $G$ be a virtually cyclic group. Then $G$ contains a unique maximal normal finite subgroup $F$ such that one of the following holds

\begin{enumerate}

\item the finite case, $G =F$;
\item the orientable case, $G/F$ is the infinite cyclic group;
\item the nonorientable case, $G/F$ is the infinite dihedral group.

\end{enumerate}

\end{lem}

Note that the above lemma implies that a torsion-free virtually cyclic group is either trivial or infinite cyclic. Thus we have the following
\begin{cor}\label{torsionfreeBVC}
Let $G$ be a torsion-free group, then $G$ has BVC if and only if there exist elements $g_1,g_2,\ldots g_n$ in $G$ such that every element in $G$ is conjugate to a power of some $g_i$.
\end{cor}

\begin{lem} \label{vircycorder}\cite[Lemma 1.6]{vPW}
Let $V$ be a virtually cyclic group and let $g,h \in V$ be two elements of infinite order, then there exist $p,q \in \BZ$ such that $g^p =h^q$. Furthermore, there exists $v_0 \in V$ such that for any $v\in V$ of infinite order there exist nonzero $p_0,p$ such that  $v_0^{p_0} = v^p$ with $\frac{p_0}{p} \in \BZ$.
\end{lem}

\begin{lem} \label{BVCfinitesubgroup}\cite[Lemma 1.7]{vPW}
If a group $G$ has BVC, then $G$ has finitely many conjugacy classes of finite subgroups. In particular, the order of finite subgroups in $G$ is bounded.
\end{lem}

In a group $G$, we call an element $g$ \emph{primitive} if it cannot be written as a proper power. Note that a primitive element is necessarily of infinite order. \Cref{torsionfreeBVC} implies the following

\begin{lem} \label{primitiveBVC}
Let $G$ be a torsion-free group. If $G$ has infinitely many conjugacy classes of primitive elements, then $G$ does not have BVC.
\end{lem}

\begin{lem} \cite[Lemma 5.6]{KMN} \label{finiteindexsubgroupbvc}
If a group $G$ has BVC, then any finite index subgroup also has BVC.
\end{lem}

\begin{defn} \label{defn-bfamily}
Let $\family F$ be a family of subgroups. For a natural number $n \geq 1$, we say that a group $G$ has property $n \mathcal F$ if there are $H_1, \ldots, H_n \in \mathcal F$ such that any cyclic subgroup of $G$ is contained in a conjugate of $H_i$ for some $i$. We say that $G$ has $b \family{F}$ if $G$ has $n \family{F}$ for some $n \in \BN$. 
\end{defn}

We are mostly interested in $b \VCYC$ as both $b\CYC$ and BVC imply $b \VCYC$, which leads to unified proofs when we deal with finiteness properties of $E_{\CYC}(G)$ and $\uu EG$. Note that for torsion-free groups all three notions agree. However, as the following two examples show, they generally do not coincide.

\begin{ex}
Consider the group $\BZ \times \BZ/2$. It is virtually cyclic, thus has $b \VCYC$ as well as BVC. But it does not have $b \CYC$. In fact, for any $n\geq 1$, let $C_n$ be the subgroup generated by $(2^n,1)$. Then for any $m\neq n$, $C_m$ cannot be conjugate to a subgroup of $C_n$. See \Cref{bc-vc} for more information.
\end{ex}

\begin{ex}\label{ex-bcyc-but-not-bvc}
Taking $G = (\BZ/2)^{\infty}$ and applying \cite[Thereom 1.1]{Os}, one embeds $G$ into a finitely generated group with only three conjugacy classes. In particular, this group has $b\CYC$ and $b \VCYC$. But it cannot have BVC, since the orders of finite subgroups are not bounded (\Cref{BVCfinitesubgroup}).
\end{ex}

The following is immediate:
\begin{lem}\label{bfquotient}
Suppose the family $\family{F}$ is closed under quotients. If $\pi \colon G \to Q$ is an epimorphism and $G$ has $b \family{F}$, then $Q$ has $b \family{F}$.
\end{lem}

\begin{lem}\label{finite-index-bf}
Let $K \subgroup G$ be a finite index subgroup and suppose $G$ has $b \family F$. Then $K$ also has $b \family F$.
\end{lem}
\begin{proof}
Let $m = [G:K]$ and let $K g_i$ for $1 \leq i \leq m$ be the right cosets. Furthermore let $\{ H_1, \ldots, H_n \}$ be a witness of $b \family F$ for $G$. Then consider the following finite collection of subgroups of $K$ which lie in $\family F$:
$$
\{ g_i H_j g_i^{-1} \cap K \mid 1 \leq j \leq n, 1 \leq i \leq m  \}
$$
We claim that these form a witness to $b \family F$ for $K$. Let $C \subgroup K$ be some cyclic subgroup, then there exists some $g \in G$ such that $C \subgroup g H_j g^{-1}$ for some $j$. Write $g = k g_i$ for some $i$ and some $k \in K$. Then $k^{-1} C k \subgroup g_i H_j g_i^{-1} \cap K$.
\end{proof}

\begin{thm}\label{bVCYC-thm}
Let $G$ be a finitely generated group in one of the following classes
\begin{enumerate}
    \item \label{bVCYC-solvable} virtually solvable groups,
    \item \label{bVCYC-one-relator} one-relator groups,
    \item \label{bVCYC-acy-hyperbolic} acylindrically hyperbolic groups,
    \item 3-manifold groups,
    \item CAT(0) cube groups.
\end{enumerate}
If $G$ has $b \VCYC$, then $G$ is virtually cyclic. 
\end{thm}
\begin{proof} 
\begin{enumerate}
    \item The proof works the same as \cite{GW2013} without much change. In fact \cite[Lemma 2.1]{GW2013} is available for $b \VCYC$ by \Cref{finite-index-bf}. \cite[Lemma 2.2]{GW2013} is still available via almost the same proof. Moreover, \cite[Lemma 2.3]{GW2013} can be easily deduced from \Cref{bfquotient}. \cite[Lemma 2.4]{GW2013} can be directly applied to the $b \VCYC$ case, since BVC is the same as $b\VCYC$ for a torsion-free group. 

    \item If $G$ contains torsion, the group is hyperbolic by Newman's Spelling Theorem  \cite{New68} and the claim follows from (c) below. Otherwise $G$ is torsion-free and the result follows from \cite[Theorem 2.12]{vPW}. One can also check that the proof of  \cite[Theorem 2.12]{vPW} indeed works for $b \VCYC$.
    
    \item The proof of \cite[Proposition 3.2]{vPW} showed that $G$ does not have $b \VCYC$. In fact, it was shown that there are infinitely many primitive conjugacy classes of elements in $G$ such that any two of them cannot be conjugated into a common virtually cyclic subgroup.
    
    \item The proof is almost the same as \cite[Proposition 3.6]{vPW}. \cite[Corollary 3.4]{vPW} now is replaced by \Cref{bfquotient} which is true for all groups.
    
    \item The proof given in \cite[Section 4]{vPW} can be carried over, observing that \cite[Lemma 4.6]{vPW} still holds with BVC replaced by $b\VCYC$.
\end{enumerate}
\end{proof}

\begin{cor}\label{BVChomology} If $G$ has $b\VCYC$ and surjects onto a finitely generated group $Q$ that lies in one of the classes described in \Cref{bVCYC-thm}, then $Q$ is virtually cyclic. In particular, the abelianization $H_1(G, \BZ)$ is finitely generated of rank at most one.
\end{cor}

\section{Conjugacy Growth of Groups and BVC for Linear Groups} \label{section-conj-growth}

In this section, we establish a connection between BVC and conjugacy growth. As an application, we show that a finitely generated linear group has BVC if and only if it is virtually cyclic. The proof is based on Breuillard--Cornulier--Lubotzky--Meiri's results \cite{BCLM} on the conjugacy growth of linear groups.

\subsection{Growth of Groups and BVC} Let $G$ be a group with a finite symmetric generating set $S$. We define the word metric on $G$ as follows
$$d_S(g,h) =\min \{n \mid g^{-1}h=s_1s_2\cdots s_n,s_i\in S\}.$$
This can also be seen as the metric on the Cayley graph $\operatorname{Cay}(G,S)$ of $G$ with respect to $S$, where we assign edges unit length. For any $g\in G$, we define the word length of $g$ via 
$$|g|_S = d_S(e,g),$$
where $e$ is the identity element of $G$. Now given $n>0$, we denote by $B_n(G,S)$ the ball of radius $n$ with respect to the word metric. The word growth function is the function that maps $n>0$ to $|B_n(G,S)|$, i.e. the number of elements of distance at most $n$ from the identity.

Similarly, for an element $g \in G$, we define its \emph{length up to conjugacy} by
$$
|g|_S^c = \min \left\{ \left|hgh^{-1} \right|_S \mid h\in G \right \}.
$$
By definition, this number only depends on the conjugacy class $[g]$ of $g$.

Now given $n>0$, we can consider the ball of radius $n$ under the length up to conjugacy in the set of conjugacy classes of $G$,
$$B^c_n(G,S) = \{[g]\mid g\in G, |g|_S^c \leq n\}$$
The conjugacy growth function $g_c(n)$ assigns to $n>0$ the number $|B_n^c(G,S)|$, i.e. the number of conjugacy classes which intersect $B_n(G,S)$.  For $f,g \colon \BN \to \BN$, we write $f \preceq g$ if there is some constant $C \in \BN$ such that $f(n) \leq g(Cn)$ for all $n \in \BN$. If $f \preceq g$ and $g \preceq f$, we say that $f$ and $g$ are equivalent and write $f \sim g$. Under this equivalence relation, the conjugacy growth function is independent of the choice of generating set. We say that a group has linear (resp. at most linear) conjugacy growth if $g_c(n) \sim n$ (resp. $g_c(n) \preceq n$), and we say that a group has exponential conjugacy growth if $g_c(n) \sim 2^n$ or equivalently if
$$\liminf_{n\rightarrow \infty} \frac{\log|B^c_n(G,S)|}{n} >0.$$
For more information about conjugacy growth, we refer to \cite{GubaSapir2010} and \cite{HuOs}. To build the connection between conjugacy growth and BVC, we need the notion of distortion, see \cite[Section 2]{CaFr} for more information.

\begin{defn} \cite[III.$\Gamma$.3.13]{BH}
Let $G$ be a finitely generated subgroup, and let $S$ be a symmetric finite generating set as above. The algebraic translation number of an element $h\in G$, denoted by $||h||_S$, is the limit
$$ ||h||_S := \lim_{n\rightarrow \infty} \frac{|h^n|_S}{n}$$
An element $h\in G$ is undistorted if the translation number $||h||_S$ is positive.
\end{defn}

Note that an element $h$ of infinite order is undistorted in $G$ if and only if the infinite cyclic subgroup $\langle h \rangle$ is quasi-isometrically embedded in $G$.

\begin{rem}\label{rem-transnumber}
The length $||h||_S$ depends only on the conjugacy class of $h$ and $||h^m|| = |m| ||h||_S$ for every $m\in \BZ$ \cite[III.$\Gamma$.3.14]{BH}. The property of being undistorted is independent of the generating set $S$\cite[Remark 2.6]{CaFr}.
\end{rem}

\begin{lem} \label{bound-word-length}
Let $G$ be a group with a finite generating set $S$. Then $$|h|_S \geq ||h||_S$$ for all $h\in G$.
\end{lem}

\Proof For any $h\in G$ we have
$$d_S(e,h^n )\leq d_S(e,h)+d_S(h,h^2)+\cdots + d_S(h^{n-1},h^n). $$
Since $d_S(h^i,h^{i+1}) = d_S(e,h)$ for any $i$, it follows 
$$\frac{|h^n|_S }{n} = \frac{d_S(e,h^n )}{n} \leq d_S(e,h)=|h|_S.$$
Letting $n$ go to infinity, we have $||h||_S \leq |h|_S$.
\qed

\begin{lem}\label{VC--undistorted-linear}
Let $G$ be a finitely generated group and $V$ be a virtually cyclic subgroup of $G$. If $V$ contains an infinite order element that is undistorted, then the conjugacy growth of $V$ in $G$ is linear, i.e. the limit
$$\lim_{n\rightarrow \infty} \frac{|B_n^c(G,S)\cap V|}{n}$$
exists and is greater than $0$. Therefore, if $G$ has $b\VCYC$ and every infinite order element of it is undistorted, then it has at most linear conjugacy growth. 
\end{lem}

\Proof Note first that a virtually cyclic group has only finitely many conjugacy classes of finite subgroups. So to count the conjugacy growth of $V$ in $G$, we only need to count the elements of infinite order in $G$.  Let $v_0$ be the infinite order element in $V$ determined by \Cref{vircycorder} and let $v$ be the element that is undistorted. Then there are nonzero integers $m_0$ and $m$ such that $v_0^{m_0} = v^m$ and $\frac{m_0}{m}\in \BZ$. In particular, we have $||v_0||_S = |\frac{m}{m_0}|||v||_S>0$. In fact, this also says that any infinite order element in $V$ is undistorted using \Cref{vircycorder}.

Now by \Cref{rem-transnumber} and \Cref{bound-word-length}, we have
$$|v_0^n|_S  \geq  ||v_0^n||_S=n||v_0||_S \, .$$ 
Since the algebraic translation number is invariant under conjugation, we have  $|gv_0^ng^{-1}|_S \geq n||gv_0g^{-1}||_S = n||v_0||_S$, 
for any $g \in G$. Thus
$$ n||v_0||_S \leq |v^n_0|_S^c \leq n|v_0|_S \, .$$
This says that the conjugacy growth of the cyclic group generated by $v_0$ in $G$ is linear. As for the whole group $V$, note that $V$ satisfies the following short exact sequence by \Cref{vcstru}
$$1 \rightarrow F \rightarrow V \rightarrow C\rightarrow 1$$
where $C$ is isomorphic to $\BZ$ or the infinite dihedral group $\BZ\rtimes \BZ/2$. In the case $C \cong \BZ\rtimes \BZ/2$, the image of all the infinite order elements under the quotient map $\pi \colon V \rightarrow C$ must lie in $ \BZ\rtimes \{0\} \leq \BZ\rtimes \BZ/2$. Hence $|\pi(v)|$ makes sense in either case. And since $|\pi(v_0)|=1$, one sees that for any infinite order element $v$, we have 
$$||v||_S = |\pi(v)|||v_0||_S \, .$$
Thus given $n$, there are at most $\frac{|F|\cdot n}{||v_0||_S}$ many conjugacy classes of elements in  $ B_n^c(G,S) \cap  V $. Therefore the conjugacy growth of $V$ in $G$ is linear. \qed

There is a notion of a semihyperbolic group which generalizes hyperbolic as well as CAT(0) groups, see \cite[Section III.$\Gamma$.4]{BH} for more details. By \cite[III.$\Gamma$.4.19]{BH}, every element of infinite order in a semihyperbolic group is undistorted, hence we have

\begin{cor}
Let $G$ be a semihyperbolic group with $b \VCYC$, then it has at most linear conjugacy growth. In particular, if it is hyperbolic, then it is virtually cyclic.
\end{cor}

\Proof We only need to show that a hyperbolic group has linear conjugacy growth if and only if it is virtually cyclic. This can be found for example in \cite[Theorem 1.1]{HuOs}, see also \cite{CoornaertKnieper}. \qed

\begin{rem}
We believe that a semihyperbolic group has at most linear conjugacy growth if and only if it is virtually cyclic. Unfortunately, we do not know how to show this.
\end{rem}

Later in \Cref{BVC-linear-positive}, we will see that if a group acts properly on a CAT(0) space via semi-simple isometries, then any infinite order element is also undistorted.

\subsection{Linear Groups and BVC}\label{linear-BVC}
Recall that a linear group is a subgroup of the general linear group $GL_d(\BK)$, where $\BK$ is a field. In \cite{GubaSapir2010} it was conjectured that a non-virtually solvable finitely generated linear group has exponential conjugacy growth, which was later proven in \cite{BCLM}. In order to show \ref{thmA} we rely on the following stronger result:

\begin{thm}\cite[Theorem 1.2]{BCLM}\label{chara-poly-growth}
For every $d$, there exists a constant $c(d)>0$ such that if $\BK$ is a field and $S$ is a finite symmetric subset of $GL_d(\BK)$ generating a non-virtually solvable subgroup, then 
$$ \liminf_{n\rightarrow \infty} \frac{1}{n} \log \chi_S(n) \geq c(d) , $$
where $\chi_S(n)$ is the number of elements in $\BK[X]$ appearing as characteristic polynomials of elements of $B_n(G,S)$.
\end{thm}

In general, an element of infinite order in a linear group may be distorted, thus we cannot apply \Cref{VC--undistorted-linear} directly.

\begin{defn}\label{absolute-value}
Let $\BK$ be a field. An absolute value on $\BK$ is a function 
$$|\cdot |:\BK\rightarrow \BR_+ =\{r\geq 0\mid r\in \BR\}$$
that satisfies the following conditions:
\begin{enumerate}
    \item $|x| = 0$ if and only if $x=0$.
    \item $|xy| = |x||y|$ for all $x,y\in \BK$.
    \item $|x+y| \leq |x|+|y| $ for all $x,y\in \BK$.
\end{enumerate}
\end{defn}

We learned the following two lemmas from Yves de Cornulier's answer on Mathoverflow\footnote{see http://mathoverflow.net/questions/143305/distortion-of-cyclic-subgroups-of-linear-groups}

\begin{lem}\label{cyclic-distortion-zero}
Let $G$ be finitely generated subgroup of $\GL_d(\BK)$ where $\BK$ is some field. Let $g \in G$ be an element such that at least one of its eigenvalues is not a root of unity. Then $g$ is undistorted. 
\end{lem}

\Proof  First of all we can assume that $\BK$ is a finitely generated field. Let $\lambda$ be an eigenvalue of $g$ which is not a root of unity.  By \cite[Lemma 4.1]{Tits}, up to replacing $g$ by its inverse, there exists an absolute value $|\cdot|$ on a field extension $\BK'$ of $\BK$ such that $|\lambda| >1$. Let $||\cdot||$ be a submultiplicative matrix norm on the vector space of $n\times n$ matrices over $\BK'$. For example, we can take 
$$ ||A|| = \max_i \Sigma_{j=1}^n |a_{ij}| $$
where $A=(a_{ij})$. Then for $A, B$ two $n \times n$ matrices over $\BK'$ we have $||AB|| \leq ||A|| \cdot ||B||$ and $|\mu| \leq ||A||$ for any eigenvalue $\mu$ of $A$.

Now, let $s_1,\dots,s_m$ be the elements of some finite symmetric generating set $S$ for $G$ and let $s = \max_{1\leq i\leq m}{||s_i||}$. If $g$ can be written as a word of length $l$ in the generators $S$, then $||g|| \leq s^l$. Hence $1<|\lambda| \leq s^l$, so $l \geq \log_s|\lambda|$.  Since $g^k$ has eigenvalue $\lambda^k$, we have 
$$ |g^k|_S \geq \log_s|\lambda|^k,$$
thus
$$||g||_S = \lim_{k\rightarrow \infty}\frac{|g^k|_S}{k} \geq \log_s|\lambda|.$$
Therefore $g$ is  undistorted.  
\qed

\begin{lem}\label{cyclic-distortion-positive}
Let $G$ be finitely generated subgroup of $\GL_d(\BK)$ where $\BK$ is a field of positive characteristic. Then any element of infinite order in $G$ is undistorted. 
\end{lem}

\Proof Let $g \in G$, by \Cref{cyclic-distortion-zero} we can assume that all eigenvalues of $g$ are roots of unity. But then $g$ must have finite order. In fact, some power $h$ of $g$ will only have eigenvalues equal to one, i.e. $h$ is unipotent in $\GL_d(\BK)$. So $(h-I)^m = 0$ for some $m$. Choose $n$ with $p^n \geq m$, where $p$ is the characteristic of $\BK$. Then
$$
0 = (h-I)^{p^n} = h^{p^n}  - I,
$$
thus $h$ has finite order, so $g$ has finite order. 
\qed

\begin{thm}\label{linear-bVCYC}
Let $G \leq GL_d(\BK)$ be a finitely generated group where $\BK$ is a field. If $G$ has BVC or more generally just $b\VCYC$, then $G$ is virtually cyclic.
\end{thm}

\Proof Since a virtually solvable group has $b\VCYC$ if and only if it is virtually cyclic by \Cref{bVCYC-thm}, we can assume that $G$ is not virtually solvable. Thus \Cref{chara-poly-growth} applies and it follows that the characteristic polynomial growth of $G$ is exponential.

If $\BK$ has characteristic zero, then by Selberg's Lemma, $G$ is virtually torsion-free. By \Cref{finiteindexsubgroupbvc}, up to replacing $G$ by a finite index subgroup, we can assume that $G$ is torsion-free. Let $g_1,g_2,\ldots,g_m, g_1', g_2', \dots, g_k'$ be generators of the infinite cyclic witnesses for $b\VCYC$, where the $g_i$ have at least one eigenvalue which is not a root of unity and all eigenvalues of the $g_j'$ are roots of unity. Note that there are only finitely many characteristic polynomials corresponding to the elements in the cyclic subgroup generated by the elements $g_j'$. Thus these elements essentially do not contribute to the characteristic polynomial growth function $\chi_{\Sigma}(n)$. By \Cref{cyclic-distortion-zero}, the elements $g_i$ are undistorted. Thus by the proof of \Cref{VC--undistorted-linear} the number of characteristic polynomials with word length $\leq n$ in $G$ is linear in $n$, contradicting \Cref{chara-poly-growth}. 

If $\BK$ has positive characteristic, by \Cref{cyclic-distortion-positive}, it follows that any infinite order element in $G$ is undistorted. Then by \Cref{VC--undistorted-linear}, $G$ has at most linear conjugacy growth, which also contradicts \Cref{chara-poly-growth}. 
\qed

Since the property $b\VCYC$ passes to quotients by \Cref{bfquotient}, we can immediately conclude that representations of finitely generated groups having BVC are rather trivial:

\begin{cor}\label{linear-rep-BVC}
Let $\varphi \colon G \rightarrow L$ be a surjective homomorphism where $G$ is a finitely generated group with BVC and $L$ is linear. Then $L$ is virtually cyclic. 
\end{cor}

\subsection{Groups Acting on CAT(0) Spaces and BVC}\label{BVC-linear-positive}

In the following we provide an alternative proof of \ref{thmA} for finitely generated linear groups over a field of positive characteristic using metric spaces of non-positive curvature.

\begin{defn}
Let $G$ be a group acting by isometries on a metric space $X$. The action is called \emph{proper} if for each $x \in X$ there exists a number $r>0$ such that the set $\{g\in G \mid B(x,r) \cap gB(x,r)\}$ is finite, where $B(x,r)$ is the ball of radius $r$ centered at $x$.
\end{defn}

\begin{defns}\cite[\nopp II.6.1]{BH} Let $X$ be a metric space and let $g$ be an isometry of $X$. The \emph{displacement function} of $g$ is the function $d_g \colon X \rightarrow \BR_+ =\{r\geq 0 \mid r\in \BR\}$ defined by $d_g(x) =d(g x,x)$. The \emph{translation length} of $g$ is the number $|g|:= \inf \{d_g (x) \mid x\in X\}$. The set of points where $d_g$ attains this infimum will be denoted by $Min(g)$. More generally, if $G$ is a group acting by isometries on $X$, then $Min(G):=\bigcap_{g\in G} Min(g)$. An isometry~$g$ is called \emph{semi-simple} if $Min(g)$ is non-empty. An action of a group by isometries of $X$ is called \emph{semi-simple} if all of its elements are semi-simple.
\end{defns}

It is clear that the translation length is invariant under conjugation, i.e. $|h g h^{-1}| = |g|$ for any $g,h \in G$. Moreover, if $G$ acts properly on a CAT(0) space via semi-simple isometries, we have that $|g^n| = |n| \cdot |g|$ for any $n \in \BZ$, e.g. by the Flat Torus Theorem \cite[\nopp II.7.1]{BH}. It turns out that the translation length can also be defined by the following limit for $g$ a semi-simple isometry
$$
|g| = \lim_{n \to \infty} \frac{d(x, g^n x)}{n},
$$
where $x$ is an arbitrary point of the CAT(0) space $X$ \cite[\nopp II.6.6]{BH}. The following example shows that this is not true if $X$ is not CAT(0).

\begin{ex}
 Take  $S^1$ to be the standard circle of radius $1$. Let the cyclic group $\langle t \mid t^n=1\rangle$ act on $S^1$ via rotation by the angle $\frac{2\pi}{n}$ for $n \geq 2$. The action is semi-simple and $|t| = \frac{2\pi}{n}$. But $\lim_{n \to \infty} \frac{1}{n} d(x, t^n x) =0$ for any $x \in S^1$.
\end{ex}

\begin{prop}\label{linear-act-building}
Let $G \leq GL_d(\BK)$ be a finitely generated subgroup where $\BK$ is a field of positive characteristic. 
Suppose that there is a bound on the orders of finite subgroups in $G$. Then $G$ acts on a CAT(0) space properly via semi-simple isometries.
\end{prop}

\Proof This is essentially due to Degrijse and Petrosyan \cite[Section 5.4]{DP15}. Note that they obtained an action of $G$ on a CAT(0) space $X$, where $X$ is a product of Euclidean buildings. Moreover, $X$ is separable and $G$ acts discretely on $X$ with countable locally finite stabilizers via semi-simple isometries \cite[II.6.6 (2)]{BH}. By our assumption on the orders of finite subgroups on $G$, it follows that all point stabilizers are finite. Hence (by for example \cite[2.2(2)]{DP15}) the action is proper.
\qed

\begin{lem} \cite[\nopp I.8.18]{BH} \label{boundwordlength}
Let $(X,d)$ be a metric space. Let $G$ be a group with a finite generating
set $S$ and associated word metric $d_S$. If $G$ acts by isometries on $X$, then for any choice of basepoint $x_0 \in X$ there exists a constant $C> 0$ such that 
$$d_X(gx_0,hx_0) \leq C d_S(g,h)$$
for all $g,h \in G$.
\end{lem}

\begin{lem}\label{proper-CAT(0)-undist}
Let $G$ be a finitely generated group acting on a CAT(0) space $X$ properly via semi-simple isometries. Then every infinite order element in $G$ is undistorted.
\end{lem}

\Proof Let $S$ be a finite symmetric generating set for $G$ and let $g \in G$ be an element of infinite order. Since the action is proper, the translation length $|g|$ is positive. By \Cref{boundwordlength}, for any fixed point $x_0$ in $X$, there exists $C>0$ such that 
$$ |g^n|_S = d_S(g^n,1) \geq C d_X(g^nx_0,x_0) \geq C|g^n|= n C|g|$$
This implies that $g$ is undistorted. \qed

Now applying \Cref{VC--undistorted-linear}, we have
\begin{cor} \label{conjugacygrowth}
Let $G$ be a finitely generated group acting on a CAT(0) space $X$ properly via semi-simple isometries. Then if $G$ has $b\VCYC$, the conjugacy growth function of $G$ is at most linear.
\end{cor}

This also leads to a different proof for \Cref{linear-bVCYC} in the positive characteristic case under the slightly stronger assumption of having BVC instead of $b\VCYC$.

\begin{thm}
Let $G \leq GL_d(\BK)$ be a finitely generated subgroup where $\BK$ is a field of positive characteristic. Then $G$ has BVC if and only if it is virtually cyclic.
\end{thm}

\Proof By \Cref{bVCYC-thm}, we can assume that $G$ is not virtually solvable. By \Cref{BVCfinitesubgroup}, if $G$ has BVC, then it has only finitely many conjugacy classes of finite subgroups. Hence by \Cref{linear-act-building}, $G$ acts properly on a CAT(0) space via semi-simple isometries. Thus, $G$ has linear conjugacy growth, which contradicts \Cref{chara-poly-growth}. \qed

\section{The Classifying Space for the Family of Cyclic Subgroups}\label{cyclic-EF}

In this section, we study the classifying space for the family of cyclic subgroups, denoted by $E_{\CYC}(G)$ for a group $G$. The results on its finiteness properties are largely analogous to those obtained for the classifying space for the family of virtually cyclic subgroups. We will first study for which groups $G$ the $G$-CW-complex $E_{\CYC}(G)$ can be of finite type. As it turns out most proofs given in the literature can be generalized slightly such that their conclusions apply to both $E_{\CYC}(G)$ as well as $\uu E G$. Second, we will deal with questions concerning the finite-dimensionality of $E_{\CYC}(G)$ and relate it to previously known results on the finite-dimensionality of $\uu E G$.

\begin{prop}\label{EC-EG-finite-type}
If a group $G$ has a model for $E_{\CYC} G$ of finite type, then it has a model for $EG$ of finite type. In particular, $G$ is finitely presented.
\end{prop}
\begin{proof} Note that there are models of finite type for $E C$ where $C$ is a cyclic group. Then the claim follows from the transitivity principle \cite[Proposition 5.1]{LuWe12}.
\end{proof}

Similar to \cite[Lemma 1.3]{vPW}, one obtains the following using the construction in \cite[Proposition 2.3]{Lu89} 

\begin{lem}
A group $G$ admits a finite $0$-skeleton for $E_{\mathcal{C}yc}G$ if and only if $G$ has $b \CYC$.
\end{lem}

\begin{ex}\label{inf-dihedral}
Let $D_{\infty} = \BZ \rtimes \BZ/2 = \langle t,s \mid s^2 = 1, sts = t^{-1} \rangle$ be the infinite dihedral group. Then $\langle t \rangle$, $\langle s \rangle$ and $\langle ts \rangle$ are witnesses to $b\CYC$ for $D_{\infty}$ since $t s t^{-1} = t^{1-2n} s$. A straightforward calculation also shows that there cannot be fewer witnesses.
\end{ex}

In \cite{GW2013} Lemma 2.2 played an important role for the proof of \ref{ConjA} for solvable groups. The statements of Lemma 2.2 hold with BVC replaced by $b\CYC$. For the convenience of the reader, we present part (a) of this lemma in a more general context.

\begin{lem} Let $G$ be a group satisfying $b \family F$ where $\family F$ is a family of Noetherian subgroups, i.e. any subgroup of an element $H \in \family F$ is finitely generated. Then $G$ satisfies the ascending chain condition on normal subgroups.
\end{lem}
\begin{proof}
The group $G$ can be written as $G = \bigcup_{i = 1}^n \bigcup_{g \in G} H_i^g$ where $\{ H_i \mid 1 \leq i \leq n \}$ is a witness to $b \family F$. But then any normal subgroup $N$ of $G$ can be likewise expressed as $N = \bigcup_{i = 1}^n \bigcup_{g \in G} (N \cap H_i)^g$.

Let $(N_j)$ be an ascending chain of normal subgroups of $G$. For any $i$, the chain $(N_j \cap H_i)_j$ has to stabilize since $H_i$ was Noetherian, i.e. there exists $j_i$ such that $N_j \cap H_i = N_{j+1} \cap H_i$ for all $j \geq j_i$. Then the original chain stabilizes at $j_{\max} = \max_{1 \leq i \leq n} j_i$.
\end{proof}

\begin{rem}
Observe that $b\CYC$ fails to pass to finite index overgroups, a counterexample can be given by $\BZ \subgroup \BZ \times \BZ/2$.
\end{rem}

Note that there is no assumption about absence of torsion in the following (cf. \Cref{primitiveBVC})

\begin{observation}\label{primitive-conj-bCYC}
If $G$ has $b\CYC$, then $G$ has only finitely many primitive conjugacy classes.
\end{observation}

\begin{lem}\label{bcyc-orient-vc}
Let $F$ be a finite group and suppose that $V = F \rtimes_{\varphi} \BZ$ is a group with $b\CYC$. Then $F = 1$.
\end{lem}
\begin{proof}
Let $d$ be the order of $\varphi$. Then $F \times d\BZ$ is a subgroup of index $d$ in $F \rtimes_{\varphi} \BZ$. Hence by \Cref{finite-index-bf}, we can assume $V= F\times \BZ$.

Now assume that $F$ is nontrivial. Let $c$ be an element of maximal order in $F$ and let $p$ be a prime that divides its order. Then for any $n\geq 1$, $g_n = (c,p^n)\in F \times \BZ$ is primitive in $F\times \BZ$. If fact, if $(x,k)^m =(x^m,mk) = (c,p^n)$, for some $m > 1$, then $p$ divides $m$. On the other hand since $x^m = c$, $c$ lies in the subgroup generated by $x$. But since $c$ has maximal order, we have  $x$ and $c$ generate the same cyclic subgroup in $F$. But since $p$ divides the order of $c$ and $m$, this cannot happen. When $n\neq m$, $g_n$ is not conjugate to $g_m$ since the second coordinate differs. Thus by \Cref{primitive-conj-bCYC}, the claim follows.
\end{proof}

\begin{prop}\label{bc-vc}
A virtually cyclic group $V$ has $b \CYC$ if and only if $V$ is finite, infinite cyclic or infinite dihedral.
\end{prop}
\begin{proof}
By \Cref{bcyc-orient-vc} and \Cref{vcstru}, the only case left to consider is if $V$ is nonorientable, i.e. there is an exact sequence
$$
1 \to F \to V \to D_{\infty} \to 1
$$
with $F$ finite. But then $V$ has a finite index subgroup isomorphic to $F \rtimes \BZ$, hence $F = 1$ and the claim follows from \Cref{inf-dihedral}.
\end{proof}

\begin{lem} \label{inf-dih-fintype}
There exists a model of finite type for $E_{\CYC} D_{\infty}$ and any model for $E_{\CYC} D_{\infty}$ has to be infinite-dimensional.
\end{lem}
\begin{proof}
Let $D_{\infty} = \BZ \rtimes \BZ/2 = \langle s,t \mid s^2 = 1, sts = t^{-1} \rangle$. We claim that the join $E = \BZ * E \BZ/2$, given an appropriate action, is a model for $E_{\mathcal{C}yc} D_{\infty}$. We write $[x,y,q]$ for an element in $E$, where $x \in \BZ$, $y \in E \BZ/2$ and $q \in [0,1]$. Note that $[x,y,0] = [x,y',0]$ and $[x,y,1] = [x',y,1]$ for all $x,x' \in \BZ$ and $y,y' \in E \BZ/2$.
We then define the action as follows:
\begin{align*}
    t \cdot [x,y,q] &= [x+2, y, q] \\
    s \cdot [x,y,q] &= [-x, s\cdot y, q] 
\end{align*}
Then one observes that the stabilizer of $[x,y,q]$ with $0 < q < 1$ is trivial. The stabilizer of $[x,y,0]$ is equal to $\langle t^x s \rangle$ and the stabilizer of $[x,y,1]$ equals $\langle t \rangle$. One furthermore checks that for $n \neq 0$
$$
E^{\langle t^n \rangle} = E \BZ/2 \simeq * 
$$
and for $n$ arbitrary 
$$
E^{\langle t^n s \rangle} = \BZ^{\langle t^n s \rangle} = \{ n \}.
$$
Since $E$ itself is contractible as well, it follows that $E$ is a model for $E_{\CYC}(G)$ of finite type.

The claim about the infinite-dimensionality of any model for $E_{\CYC}(G)$ follows from \Cref{maximal-cyclic-ecyc-infinite-dim} below by noting that $\langle t \rangle$ is a normal maximal cyclic subgroup of $D_{\infty}$. Alternatively observe that $E/D_{\infty}$ is homotopy equivalent to the suspension of $\BR P^{\infty}$.
\end{proof}

\begin{cor}\label{vc-finite-type}
Let $G$ be a virtually cyclic group, then it has a model for $E_{\CYC}(G)$ of finite type if and only if it is finite, infinite cyclic or infinite dihedral.
\end{cor}

\Proof By \Cref{bc-vc}, we only need to prove that there is a model of finite type for $E_{\CYC}(G)$ if $G$ is finite, infinite cyclic or infinite dihedral. If $G$ is a finite group, then the standard bar-construction provides such a model. If $G$ is infinite cyclic, we can take $E_{\CYC}(G)$ to be a point. In case $G$ is infinite dihedral \Cref{inf-dih-fintype} provides a model of finite type. \qed

\textbf{Proof of \ref{thm:ecyc-g-conjecture}}:  By \Cref{bVCYC-thm} and \Cref{vc-finite-type} and the fact that both $b \CYC$ and $b\VCYC$ implies BVC, we only need to show if an elementary amenable group has a finite type model for $E_{\CYC}(G)$, then it is virtually cyclic. Since Kropholler's class of groups $\operatorname{LH}\mathfrak{F}$ contains the class of elementary amenable groups, \cite[Theorem B]{Kro93} implies that an elementary amenable group of type $FP_{\infty}$ has a bound on the orders of its finite subgroups and finite Hirsch length. Hence by 
\Cref{EC-EG-finite-type}, $G$ cannot have a model of finite type for $E_{\CYC}(G)$ if it has infinite Hirsch length. When $G$ has finite Hirsch length, by \cite{HiLi}, it is locally finite by virtually solvable, i.e. it has a virtually solvable quotient with locally finite kernel. Thus by Kropholler's result, $G$ is finite by virtually solvable. Now by \Cref{bfquotient} and \Cref{bVCYC-thm}~(a), $G$ is finite by virtually cyclic. Hence $G$ is virtually cyclic.  \qed

Even though $\uu EG$ is conjecturally never of finite type except in trivial cases, finite-dimensional models for $\uu EG$ abound for reasonable classes of groups. For example, finite-dimensional models for $\uu E G$ have been constructed for hyperbolic groups \cite{JL}, CAT(0) groups \cite{Lueck2009}, elementary amenable groups of finite Hirsch length \cite{FluchNucinkis2013, DegrijsePetrosyan2013} and also certain linear groups \cite{DegrijseKoehlPetrosyan} such as discrete subgroups of $\GL_d(\BR)$. The question arises in which cases we can expect a finite-dimensional model for $E_{\CYC}(G)$. We will first settle the question in the case that $G$ is virtually cyclic and then draw a couple of immediate conclusions in the general case. Let us first record the following simple observation.

\begin{observation}\label{restricting-ecyc-model}
Let $H$ be a subgroup of a group $G$ and let $X$ be a model for $E_{\CYC}(G)$, then $\res_H^G X$ is a model for $E_{\CYC}(H)$. 
\end{observation}

Recall that the classifying space $EF$ of a non-trivial finite group $F$ cannot be finite-dimensional \cite[\nopp VIII.2.5]{Brown1982}.

\begin{lem}\label{maximal-cyclic-ecyc-infinite-dim}
Suppose $H \subgroup G$ is a maximal cyclic subgroup and assume that $[N_G(H) : H]$ is finite but not equal to one, where $N_G(H)$ is the normalizer of $H$ in $G$. Then any model for $E_{\CYC}(G)$ has to be infinite-dimensional.
\end{lem}
\begin{proof}
Let $X$ be a model for $E_{\CYC}(G)$. Since $H$ is cyclic, the CW-complex $X^H$ is contractible. Observe that all isotropy groups of $X^H$ are equal to $H$ since $H$ was maximal cyclic. This implies that the Weyl group $N_G(H)/H$ of $H$ acts freely on $X^H$. Since $N_G(H)/H$ is non-trivial finite, $X^H$ has to be infinite-dimensional.
\end{proof}

\begin{prop} \label{finite-group-finite-dim-cyc-model}
Let $G$ be a finite group with a finite-dimensional model for $E_{\CYC}(G)$. Then $G$ is already cyclic.
\end{prop}
\begin{proof}
By \Cref{restricting-ecyc-model} we only need to consider minimal non-cyclic groups, i.e. finite groups such that every proper subgroup is cyclic. Thus by \cite{MillerMoreno1903} $G$ is solvable. Moreover, since the proper subgroup $[G,G]$ has to be cyclic, $G$ has derived length at most~$2$. Let $H$ be a maximal cyclic subgroup containing $[G,G]$, then $N_G(H)=G$. Then by \Cref{maximal-cyclic-ecyc-infinite-dim}, $H=G$ and hence $G$ is cyclic.
\end{proof}

\begin{prop}\label{vc-cy-fd-model}
Let $V$ be virtually cyclic. Then $E_{\CYC} V$ is finite-dimensional if and only if $V$ is cyclic.
\end{prop}
\begin{proof}
By \Cref{finite-group-finite-dim-cyc-model} we only need to prove the claim if $V$ is infinite. Suppose $V$ is orientable, i.e. $V \cong F \rtimes_{\varphi} \BZ$ for some finite group $F$ and some $\varphi \in \Aut(F)$ and assume that $E_{\CYC} V$ is finite-dimensional. If $n$ denotes the order of $\varphi$ then $F \times \BZ \cong F \rtimes_{\varphi} n \BZ$ and by \Cref{restricting-ecyc-model} also $F \times \BZ$ has a finite-dimensional classifying space. But $\BZ \subgroup F \times \BZ$ is a normal maximal cyclic subgroup, thus $F = 1$ by \Cref{maximal-cyclic-ecyc-infinite-dim}.

Now, suppose $V$ was non-orientable having a finite-dimensional model for $E_{\CYC}V$. Let $F$ be the maximal normal finite subgroup of $V$, then $F \rtimes \BZ$ is a subgroup of $V$. By the above, it follows that $F = 1$, so $V$ is infinite dihedral. But this is impossible by \Cref{inf-dih-fintype}.
\end{proof}

\begin{cor}
A virtually cyclic group $V$ has  a finite or finite-dimensional model for $E_{\CYC}V$ if and only if $V$ is cyclic.
\end{cor}

From \Cref{vc-cy-fd-model} we immediately obtain the following
\begin{observation}\label{cyc-vc-fd-model}
If $G$ is a group having a finite-dimensional model for $E_{\CYC}(G)$, then there is a finite-dimensional model for $\uu E G$.
Conversely, suppose that $G$ is a group having a finite-dimensional model for $\uu E G$. Then $G$ admits a finite-dimensional model for $E_{\CYC}(G)$ if and only if $\CYC(G) = \VCYC(G)$.
\end{observation}

Obviously the condition $\CYC(G) = \VCYC(G)$ holds whenever the group $G$ is torsion-free. However, this is not necessary, even for virtually free groups. For example, groups of the form $G = \ast_{i = 1}^n \BZ/n_i\BZ$ where $n_i \geq 0$ admit a finite-dimensional model for $E_{\CYC}(G)$ if and only if $n_i \neq 2$ for all $i$ or $G \cong \BZ/2$ by the Kurosh subgroup theorem.

\begin{lem}
Let $A$ be an abelian group with $E_{\CYC}(A)$ finite-dimensional. Then $A$ is cyclic, torsion-free or locally finite cyclic.
\end{lem}
\begin{proof} 
By \Cref{finite-group-finite-dim-cyc-model} we can assume in the following that $A$ is infinite. If $A$ is finitely generated, we can write $A \cong \BZ^n \times F$ with $F$ finite abelian and $n \geq 1$. In particular, $A$ contains $\BZ \times F$, so $F = 1$, i.e. $A$ is torsion-free.
More generally, if $A$ contains an element of infinite order $x$, then any finite set $\{ y_1, \ldots, y_n \} \subset A$ together with the element $x$ will generate an infinite abelian subgroup, which must be torsion-free by the previous observation. The only case that remains is $A$ being an infinite torsion group. But since any finite subgroup has to be cyclic, it follows that $A$ is locally finite cyclic.
\end{proof}

For example, the Pr\"ufer $p$-group $P$ is a countable locally cyclic infinite abelian $p$-group. By \cite[Theorem 4.3]{LuWe12} it follows that the minimal dimension of a model for $E_{\CYC}(P)$ equals one.

\begin{lem}
Let $G$ be elementary amenable and suppose that there is a finite-dimensional model for $E_{\CYC}(G)$. Then $G$ is virtually solvable.
\end{lem}
\begin{proof}
Since $E_{\CYC}(G)$ has a finite-dimensional model, so does $E_{\FCYC}(G)$ by \cite[Proposition 5.1]{LuWe12}, where $\FCYC$ denotes the family of finite cyclic subgroups. It follows that the Hirsch length $\operatorname{h}(G)$ of $G$ is finite, since $\operatorname{h}(G) \leq \cd_{\BQ}(G) \leq \operatorname{gd}_{\FCYC}(G) < \infty$. The first inequality follows from \cite[Lemma 2]{Hillman1991}. For the second inequality note that $\BQ[G/F]$ is a projective $\BQ G$-module for $F$ finite \cite[I.8 Ex. 4]{Brown1982} and thus the cellular chain complex of $E_{\FCYC}(G)$ yields a projective resolution of $\BQ$ over $\BQ G$. Moreover, note that any locally finite subgroup $H$ of $G$ has to be locally cyclic, in particular $H$ is abelian. Combined with the structure theorem of elementary amenable groups of finite Hirsch length \cite{HiLi}, it follows that $G$ is virtually solvable.
\end{proof}

\begin{lem}
A hyperbolic group $G$ has a finite-dimensional classifying space $E_{\CYC}(G)$ if and only if it does not contain the infinite dihedral group as a subgroup and the normalizers of all non-trivial finite subgroups are finite cyclic. Moreover, if a subgroup $K$ is maximal finite, then $N_G(K) = K$.
\end{lem}
\begin{proof}
Note that a hyperbolic group $G$ has a finite-dimensional model for $\uueg$ \cite[Remark 7 + Proposition 8]{JL}. Hence by \Cref{cyc-vc-fd-model}, $G$ has a finite-dimensional classifying space $E_{\CYC}(G)$ if and only if $\VCYC(G) = \CYC(G)$.

Suppose first that $E_{\CYC}(G)$ is finite-dimensional, which is equivalent to saying that $\VCYC(G) = \CYC(G)$. In particular, the infinite dihedral group does not appear as a subgroup. Now let $K$ be a non-trivial finite subgroup of $G$. If $N_G(K)$ was infinite, it would contain an element $g \in N_G(K)$ of infinite order. But then $\langle g, K \rangle \subgroup N_G(K)$ would be an infinite virtually cyclic group that is not cyclic. Thus $N_G(K)$ is finite and thus also cyclic. The claim about maximal finite subgroups being self-normalizing follows from \Cref{maximal-cyclic-ecyc-infinite-dim}.

For the converse, let $V \subgroup G$ be a non-trivial virtually cyclic subgroup. If $V$ is finite, it is cyclic, being the subgroup of its finite cyclic normalizer. Now suppose that $V$ is infinite and let $K$ denote its unique maximal normal finite subgroup. Since $K$ is normal in $V$, $N_G(K)$ is infinite. It follows that $K$ must be trivial. Since $V$ cannot be the infinite dihedral group, it follows that $V$ is infinite cyclic.
\end{proof}

Note that in the proof we have only used hyperbolicity of $G$ to conclude that it has a finite-dimensional model for $\uueg$ and any infinite subgroup of it contains an element of infinite order.

\begin{prop}
Let $G$ be a finitely generated linear group over a field of characteristic zero. If $G$ has a finite-dimensional model for $E_{\CYC}(G)$, then the virtual cohomological dimension of $G$ is finite. More precisely,
$$
\operatorname{vcd}(G) \leq \operatorname{gd}_{\CYC}(G) + 1.
$$
\end{prop}
\begin{proof}
Let $n = \operatorname{gd}_{\CYC}(G)$ be the minimal dimension of a model for $E_{\CYC}(G)$. Moreover let $H \subgroup G$ be a torsion-free subgroup of finite index and note that $E_{\CYC}(G)$ is also a model for $E_{\CYC}(H)$ via restriction. By \cite[Proposition 5.1]{LuWe12} there exists a model for $EH$ of dimension $n + 1$, since $E\BZ$ has a one-dimensional model.
\end{proof}

Linear groups of finite virtual cohomological dimension have been characterized by Aplerin and Shalen in terms of the Hirsch lengths of their finitely generated unipotent subgroups \cite{AlperinShalen1982}.

\section{Groups with Wild Inner Automorphisms and BVC}
\label{fg-group-BVC}

As noted in the introduction, essentially all proofs of \ref{ConjA} rely on understanding the BVC property for manageable classes of groups. In this section we want to construct groups outside of these realms that exhibit wild behaviour with respect to the BVC property. For constructing finitely generated examples we will make use of small cancellation theory over relatively hyperbolic groups as developed in \cite{Os} and \cite{HuOs}. In the following we content ourselves with setting up some notation that is being used later on without repeating the definition of relatively hyperbolic groups \cite{Osin2006}. Note that in this section we will use the convention that $g^w = w^{-1} g w$.

If $G$ is a group that is hyperbolic relative to a family $\{H_\lambda\}_{\lambda \in \Lambda}$ of subgroups, we call an element $g \in G$ parabolic, if it is conjugate to an element lying in one of the parabolic subgroups $H_\lambda$. Non-parabolic elements of infinite order are called loxodromic. For $g \in G$ a loxodromic element, there exists a unique maximal virtually cyclic subgroup $E_G(g)$ that contains $g$. It is given by
$$
E_{G}(g) = \{ h \in G \mid \exists~ m \in \BN \setminus \{0\} \text{ such that } h^{-1} g^m h = g^{\pm m} \}.
$$

Recall that two elements $a, b \in G$ are called \emph{commensurable} if there exists $k,l \in \BZ \setminus \{ 0 \}$ such that $a^k$ is conjugate to $b^l$. A subgroup $S$ of $G$ is called  suitable if there exist two loxodromic elements $s_1, s_2 \in S$ that are not commensurable such that $E_G(s_1) \cap E_G(s_2) = 1$.

Given a group $H$ and an isomorphism $\theta \colon A \rightarrow B$ between two subgroups $A$ and $B$ of $H$, we can define a new group $H\ast_{\theta} = H\ast_{A^t =B}$, called the HNN extension of $H$ along $\theta$, by the presentation $\langle H,t \mid t^{-1}xt =\theta(x), x\in A \rangle $. The letter $t$ is called stable letter.

A sequence $g_0 , t^{\epsilon_1} , g_1 , \ldots, t^{\epsilon_n}, g_n$ of elements with $g_i \in H$ and $\epsilon_i \in \{-1,+1\}$ is said to be \emph{reduced} if there is no pinch, where we define a pinch to be a consecutive sequence $t, g_i, t^{-1}$ with $g_i \in B$ or $t^{-1}, g_j,t$ with $g_j \in A$. Britton's Lemma states that if the sequence $g_0 , t^{\epsilon_1} , g_1 , \ldots, t^{\epsilon_n} , g_n $ is reduced and $n\geq 1$, then $g_0  t^{\epsilon_1}  g_1  \cdots t^{\epsilon_n}  g_n \neq 1$ in $H\ast_{\theta}$. The number of occurrences of $t$ or $t^{-1}$ in a reduced expression for an element $g \in H\ast_{\theta}$ is called its $t$-length, it is independent of the choice of reduced word.

We will make use of the following two lemmas repeatedly. For the convenience of the reader, we recall their statements:

\begin{lem}[{\cite[Theorem 1.4]{Osin2006}}] \label{malnormal-parabolics} Let $G$ be a group hyperbolic relative to a collection of subgroups $\{ H_\lambda \}_{\lambda \in \Lambda}$. Then for every $\lambda \in \Lambda$ and $g \in G \setminus H_{\lambda}$ the intersection $H_{\lambda} \cap H_{\lambda}^g$ is finite.
\end{lem}

In particular, if $G$ is torsion-free and $f,g \in H_{\lambda}$ are not conjugate in $H_{\lambda}$, then they are not conjugate in $G$.

\begin{lem}[{\cite[Lemma 2.14]{HuOs}}] \label{hnn-conjugacy-primitivity} Let $G$ be a group and $A, B$ two isomorphic subgroups. Let $f \in G$ be an element that is not conjugate to any element of $A \cup B$. Then in the HNN extension $G*_{A^t = B}$
\begin{enumerate}
    \item $f$ is conjugate to some $g \in G$ in $G*_{A^t = B}$ if and only if $f$ and $g$ are conjugate in $G$.
    \item If $f$ is primitive in $G$, then it is primitive as an element of $G*_{A^t = B}$.
\end{enumerate}
\end{lem}

\begin{lem} \label{rel-hyperbolic-primitive}
Let $G$ be hyperbolic relative to a subgroup $H$ and suppose that $G$ is torsion-free. Let $h \in H$ be primitive as an element of $H$. Then $h$ is primitive in $G$.
\end{lem}
\begin{proof}
First note that non-trivial powers of loxodromic elements are loxodromic again. In fact, let $g$ be loxodromic and suppose that $g^n$ is parabolic for some $n \neq 0$, i.e. $g^n = \alpha y \alpha^{-1}$, where $y \in H$ and $\alpha \in G$. Then it would follow that $0 = \tau^{\mathrm{rel}}(y) = \tau^{\mathrm{rel}}(g^n) = |n| \tau^{\mathrm{rel}}(g) > 0$ by \cite[Lemma 4.24, Theorem 4.25]{Osin2006}, where $\tau^{\mathrm{rel}}(g)$ denotes the relative translation number of $g$.

So if $h = g^n$ for some $g \in G$ and $n \geq 1$, then $g$ has to be parabolic, i.e. $g = \alpha x \alpha^{-1}$ for some $x \in H$ and $\alpha \in G$. By \Cref{malnormal-parabolics} it follows that $\alpha \in H$, thus $n = 1$, since $h$ was primitive in $H$.
\end{proof}

\begin{lem}\label{hnn-extension-primitive}
Let $G$ be a group and let $a,b \in G$ be two primitive elements. Then any primitive element $g \in G$, considered as an element of the HNN extension $G*_{a^t = b}$ is still primitive.
\end{lem}
\begin{proof} Let $g \in G$ be a primitive element. If $g$ is neither conjugate to an element of $\langle a \rangle$ nor $\langle b \rangle$, then \Cref{hnn-conjugacy-primitivity} implies that $g$ is primitive as an element of the HNN extension. So we can assume without loss of generality that $g = a$. Let $a = \omega^n$ with $n \geq 1$, where $\omega$ is of $t$-length 2, so $\omega = g_0 t^\epsilon g_1 t^{-\epsilon} g_2$ for $g_0,g_1,g_2 \in G$ and $\epsilon \in \{\pm 1\}$. For the equality $a = \omega^n$ to hold, the expression $t^{-\epsilon} g_0 g_2 t^{\epsilon}$ must be a pinch, so say $g_2 g_0 = b^m$ for some $m$. Then $t^{-\epsilon} g_2 g_0 t^{\epsilon} = a^m$. Expanding the power of $\omega$, we obtain
$$
\omega^n = g_0 t^{\epsilon} g_1 (a^m g_1)^{n-1} t^{-\epsilon} g_2
$$
This implies that $g_1 (a^m g_1)^{n-1} \in \langle a \rangle$, say $g_1 (a^m g_1)^{n-1} = a^k$. Thus
$$
a = \omega^n = g_0 b^k g_2 = g_0 b^k g_2 g_0 g_0^{-1} = g_0 b^{k+m} g_0^{-1}
$$
But since $a$ is primitive in $G$, it follows that $k+m \in \{\pm 1\}$. The equation $g_1 (a^m g_1)^{n-1} = a^k$ is equivalent to $(a^m g_1)^n = a^m a^k  = a^{m+k}$, which implies that $n = 1$, again since $a$ is primitive.

The induction step works as in the proof of \cite[Lemma 2.14]{HuOs}.
\end{proof}

\begin{lem}\label{hnn-extension-primitive-positive-powers}
Let $G$ be a group and let $a,b \in G$ be arbitrary non-trivial elements. Then any primitive element $g \in G$, considered as an element of the HNN extension $G*_{(a^n)^t = b^m}$ is primitive as long as $|n|, |m| \geq 2$.
\end{lem}
\begin{proof}
Let $g \in G$ be a primitive element. Then $g$ is neither conjugate to an element of $\langle a^n \rangle$ nor $\langle b^m \rangle$. Thus the claim follows from \Cref{hnn-conjugacy-primitivity}.
\end{proof}

The next proposition shows that the converse of \Cref{finiteindexsubgroupbvc} does not hold.

\begin{prop}\label{finite-index-over-bvc}
There exists a finitely generated torsion-free group $G$ with a finite index subgroup $H$ such that $H$ has $b\CYC$, but $G$ does not.
\end{prop}

\begin{proof} The proofs of Lemma 7.1 and Theorem 7.2 in \cite{HuOs} apply. By \Cref{hnn-conjugacy-primitivity} it follows that the elements $a f$ with $f \in F$ in the statement of Lemma 7.1 are primitive in $C$. In the last part of the proof of Theorem 7.2, one only needs to observe that the elements $a f$ with $f \in F$ are primitive in $G(i)$ by \Cref{rel-hyperbolic-primitive} since they lie in the parabolic subgroup $C$. Since $a f_1$ and $a f_2$ lie in different conjugacy classes for any $f_1 \neq f_2$  in $F$ (this was stated at the end of the proof of \cite[Theorem 7.2]{HuOs}), $G$ does not have $b\CYC$ by \Cref{primitiveBVC}.
\end{proof}

\begin{lem}\label{embedding-conjugacy-theorem}
Let $H$ be a torsion-free countable group. There exists a 2-generated torsion-free group $G$ that contains $H$ as a subgroup such that
\begin{enumerate}
    \item Any $g \in G$ is conjugate to an element of $H$.
    \item If $h \in H$ is primitive, then it is primitive as an element of $G$.
    \item If $h, h' \in H$ are not conjugate, then they are not conjugate in $G$.
\end{enumerate}
\end{lem}
\begin{proof}
Let
$$
G(0) = H * F(x,y)
$$
and enumerate elements of $H = \{ 1 = h_0, h_1, h_2, \ldots \}$ as well as $G(0) = \{ 1 = g_0, g_1, g_2, \ldots \}$.

Suppose the group $G(i)$ has been constructed such that $G(i)$ is hyperbolic relative to $H$, $\langle x, y \rangle$ is a suitable subgroup and $h_1, \ldots, h_i$ lie in $\langle x, y \rangle$ and $g_j$ for $1 \leq j \leq i$ is conjugate to an element of $H$.  Construct $G(i+1)$ out of $G(i)$ as follows: If $g_{i+1}$ is parabolic, then let $G'(i) = G(i)$, otherwise let $\iota \colon E_{G(i)}(g_{i+1}) \to \langle h_1 \rangle$ be an isomorphism and form
$$
G'(i) = \langle G_i, t \mid e^t = \iota(e) \text{ for } e \in E_{(G(i)}(g_{i+1}) \rangle
$$

By \cite[Corollary 2.16]{HuOs}, $G'(i)$ is still hyperbolic relative to $H$ and $\langle x, y \rangle$ is again suitable. Now apply \cite[Theorem 6.2]{HuOs} to the suitable subgroup $\langle x, y \rangle$, words $\{ h_{i+1}, t \}$ resp. $\{ h_{i+1} \}$ to obtain $G(i+1)$. Observe that there is a canonical quotient map $G(i) \to G(i+1)$. Here we do not distinguish $H$ and its image in $G'(i)$ or $G(i+1)$. Note that $G(i+1)$ is also hyperbolic relative  to $H$, $\langle x, y \rangle$ is again suitable.

Letting $G$ be the direct limit of the $G(i)$, it follows that $G$ will be two-generated and any element of $G$ will be conjugate to an element of $H$. Moreover, if $h \in H$ is primitive, then it will remain primitive in $G(i)$ by \Cref{rel-hyperbolic-primitive}  for any $i$, thus it will be primitive in $G$. By \Cref{malnormal-parabolics} non-conjugate elements in $H$ remain non-conjugate in $G$.
\end{proof}

\begin{prop} \label{group-three-conjugacy-classes-one-primitive}
There exists a finitely generated torsion-free group $G$ which has exactly three conjugacy classes $\{ (1), (x), (x^2) \}$, where $x \in G$ is a primitive element.
\end{prop}
\begin{proof} We first construct a countable group as follows. We let $G_0 = \langle x \rangle \cong \BZ$. Of course, the only primitive elements in $G_0$ are $x$ and $x^{-1}$. Suppose we have already constructed a chain $G_0 \subgroup G_1 \subgroup  \ldots \subgroup G_n$ of countable torsion-free groups such that the element $x$, viewed as an element of $G_n$, is primitive. To construct $G_{n+1}$ out of $G_n$, we first enumerate all primitive elements of $G_n$ by $\{p_1, p_2, \ldots \}$ and enumerate all non-trivial elements that are non-primitive by $\{ g_1, g_2, \ldots \}$. Secondly, we form the multiple HNN extension
$$
G_{n+1} = \langle G_n, \{s_i\}_{i \in \BN}, \{t_i\}_{i \in \BN} \mid p_i^{s_i} = x, g_i^{t_i} = x^2 \rangle.
$$
By an inductive application of \Cref{hnn-extension-primitive} and \Cref{hnn-extension-primitive-positive-powers} it follows that $x$ remains primitive in $G_{n+1}$.  Finally, we let $G = \bigcup_{n \geq 0} G_n$. By construction, $G$ satisfies our desired properties. Note that since $x$ is primitive, $x$ and $x^2$ are non-conjugate. To obtain a finitely generated example we can apply \Cref{embedding-conjugacy-theorem}.
\end{proof}

\begin{rem}
If we let $G$ be a group as constructed in \Cref{group-three-conjugacy-classes-one-primitive} and $x \in G$ a primitive element, then $G \times \BZ$ does not have $b\CYC$, since it has infinitely many primitive conjugacy classes $\{ (x, n) \mid n \in \BZ \}$. On the other hand, $G \times \BZ$ has $b \CYC$ if $G$ is a torsion-free group with exactly two conjugacy classes.
\end{rem}

As we have noted, if $G$ is a torsion-free group with exactly two conjugacy classes $G \times \BZ$ has $b\CYC$. However, the situation changes if we allow for a semidirect product:

\begin{prop}\label{two-conjugacy-classes-semi-bz-no-bcyc}
There exists a countable torsion-free group $H$ with exactly two conjugacy classes such that a certain extension $H \rtimes \BZ$ does not have $b\CYC$.
\end{prop}
\begin{proof}
Let $G_0 = \langle a,b\rangle$ be a free group of rank 2, let $\varepsilon_0 \colon G_0 \to \BZ$ be defined by mapping $a$ to $0$ and $b$ to $1$. Note that for any $m \in \BZ$, the element $ab^m$ is primitive and $\varepsilon_0(ab^m) = m$.

Suppose $G_n$ and $\varepsilon_n \colon G_n \to \BZ$ have been constructed. To obtain $G_{n+1}$, enumerate all non-trivial elements of $\ker \varepsilon_n$ by $\{ g_1, g_2, \ldots \}$ and form the multiple HNN extension
$$
G_{n+1} = \langle G_n, \{ t_i \}_{i \in \BN} \mid g_i^{t_i} = a \rangle.
$$
We can extend $\varepsilon_n$ to $G_{n+1}$ to define $\varepsilon_{n+1} \colon G_{n+1} \to \BZ$ by arbitrarily assigning a value to the stable letters $t_i$. Now note that for $m \neq 0$, $ab^m \in G_0 \subgroup G_n$ is neither conjugate to an element of $\langle g_i \rangle$, nor to an element of $\langle a \rangle$. Thus, by \Cref{hnn-conjugacy-primitivity}, the elements $ab^m$ are primitive in $G_{n+1}$ for $m\neq 0$.

Let $G$ be the direct limit of the $G_n$, and let $\varepsilon \colon G \to \BZ$ be induced by the epimorphisms $\varepsilon_n$. Any non-trivial element of $\ker(\varepsilon)$ is then conjugate to $a$. However, since for $m\neq 0$, the elements $ab^m$ are primitive and obviously in different conjugacy classes as $\varepsilon(ab^m)=m$, it follows that $G$ does not have $b\CYC$.
\end{proof}

\begin{prop} 
For any $m \geq 1$, there exists a finitely generated torsion-free group $G$ that has exactly $n+1$ conjugacy classes $(1), (x_1), \ldots, (x_m)$ such that $x_i^k$ is conjugate to $x_i$ for any $i$ and any $k \neq 0$.
\end{prop}
\begin{proof} Note that any torsion-free group with exactly two conjugacy classes will have the property that $x^k$ is conjugate to $x$ as long as $x$ is non-trivial. So we will demonstrate the claim for $m = 2$, the general case follows from an analogous argument. Let
$G_{-1} = \langle a,b \rangle$ be a free group of rank two, and define
$$G_0 = \langle G_{-1}, \{s_i\}_{i \in \BZ \setminus \{0\}}, \{t_i\}_{i \in \BZ \setminus \{0\}} \mid (a^i)^{s_i} = a, (b^i)^{t_i} = b \rangle. $$
Now observe that the elements $a$ and $b$ are not conjugate in $G_0$ by repeated application of \Cref{hnn-conjugacy-primitivity}. If we form an HNN extension with relation $(a^i)^{s_i} = a$ we apply \Cref{hnn-conjugacy-primitivity} to the element $b$. For relations of the type $(b^i)^{t_i} = b$ we apply the same lemma to the element $a$.

We proceed constructing countable torsion-free groups $G_n$ for $n \geq 1$ inductively. First, observe that we can write $G_n \setminus \{1 \} = S_0 \sqcup S_a \sqcup S_b$, where $S_a$ resp. $S_b$ are those elements of $G_n$ which have a non-trivial power that is conjugate to $a$ resp. $b$, and $S_0$ being defined as the complement of $S_a \cup S_b$. Note that $S_a \cap S_b = \emptyset$: If $g \in S_a \cap S_b$, then $g^k \sim a$ and $g^l \sim b$ for some $k, l \neq 0$. But then $g^{kl} \sim a^l \sim a$, and at the same time $g^{kl} \sim b^k \sim b$. But this is impossible since $a$ and $b$ are not conjugate in $G_n$ by induction. Our construction proceeds in two steps:

\textbf{Step 1.} Enumerate all element of $S_a \cup S_b = \{ g_1, g_2, \ldots \}$. We form the multiple HNN extension
$$
Q = \langle G_n, \{ t_i \}_{i \in \BN} \mid g_i^{t_i} = b \text{ if } g_i \in S_b \text{ , otherwise } g_i^{t_i} = a \rangle.
$$
In other words $Q$ is the direct limit of a sequence of HNN extensions $Q_0 \subgroup Q_1 \subgroup Q_2 \subgroup \ldots$ where $Q_0 = G_n$ and 
$$
Q_{i} = \langle Q_{i-1}, t_i \mid g_i^{t_i} = b \text{ if } g_i \in S_b \text{ , otherwise } g_i^{t_i} = a \rangle.
$$

Now we prove by induction that $a$ and $b$ are not conjugate in $Q_i$ for any $i$. Suppose the claim is true for $Q_{i-1}$. If $g_i \in S_a$ then apply \Cref{hnn-conjugacy-primitivity} to the element $b$. Note that $b$ is not conjugate to a power of $a$ in $Q_{i-1}$ by induction. Moreover $b$ is not conjugate to $g_i^n$ for any $n \neq 0$ in $Q_{i-1}$. Otherwise, it would follow that $b \sim b^k \sim g_i^{nk} \sim a^n \sim a$ in $Q_{i-1}$ since there is some $k \neq 0$ such that $g_i^k \sim a$. Interchanging the roles of $a$ and $b$, we see that the same conclusion holds if $g_i \in S_b$. Hence it follows that $a$ and $b$ are non-conjugate in $Q$.

\textbf{Step 2.} Enumerate all elements of $S_0 = \{ h_1, h_2, \ldots \}$. Again we construct a sequence of HNN extensions $P_0 \subgroup P_1 \subgroup \ldots $, starting with $P_0 = Q$. Suppose $P_{i-1}$ has been constructed and $a,b$ are non-conjugate in $P_{i-1}$. To form $P_i$, we consider the following cases:
\begin{enumerate}
    \item If a non-trivial power of $h_i$ is conjugate to $a$ in $P_{i-1}$, we form the HNN extension
    $$
    P_i = \langle P_{i-1}, s_i \mid h_i^{s_i} = a \rangle.
    $$
    We can again employ \Cref{hnn-conjugacy-primitivity} to the element $b$ to prove that $a$ and $b$ are non-conjugate in $P_i$ using the same argument as in step 1.
    
    \item If a non-trivial power of $h_i$ is conjugate to $b$ in $P_{i-1}$, we let
    $$
    P_i = \langle P_{i-1}, s_i \mid h_i^{s_i} = b \rangle.
    $$
    Again, interchanging the roles of $a$ and $b$ one sees that these two elements remain non-conjugate in $P_i$.
    
    \item In the remaining case we can choose
    $$
    P_i = \langle P_{i-1}, s_i \mid h_i^{s_i} = a \rangle.
    $$
    and observe that \Cref{hnn-conjugacy-primitivity} can be applied.
\end{enumerate}
We let $G_{n+1} = \bigcup_{i \geq 0} P_i$. Note that $a$ and $b$ are non-conjugate in $G_{n+1}$ and all elements of $G_n \setminus \{ 1 \}$ are either conjugate to $a$ or $b$ in $G_{n+1}$.

Finally, we define $G = \bigcup_{n \geq 0} G_n$. Again, \Cref{embedding-conjugacy-theorem} yields a finitely generated example.
\end{proof}

We see in particular that for a group $G$ as constructed in the previous proposition, the group $G \times \BZ$ has $b\CYC$. 

\begin{ex}
Letting $G$ be an infinite 2-generated group of exponent $p$ with exactly $p$ conjugacy classes (for example, as constructed in \cite[Theorem 41.2]{Olshanskij1991}), then the group $G \times \BZ$ does not have $b\CYC$. This can be seen by considering the elements $(x,p^n)$ for $n\in \BN$ where $x \in G$ is non-trivial.
\end{ex}

It is easy to see that a torsion-free group with infinitely many commensurability classes cannot have $b\CYC$. The following shows that a converse to \Cref{primitiveBVC} does not hold: 

\begin{prop}
There exists a finitely generated torsion-free group $G$ without primitive elements that does not have $b\CYC$.
\end{prop}
\begin{proof}
We first prove that there exists a torsion-free countable group without primitive elements and infinitely many commensurability classes. We start by an inductive procedure, letting $G_0 = F_2$ be the non-abelian free group on two generators. In fact, any countable torsion-free group with infinitely many commensurability classes of elements would work as well. Suppose $G_n$ has been constructed, enumerate all non-trivial elements $G_n \setminus \{1\} = \{g_1, g_2, \ldots \}$ and define $G_{n+1}$ as the following multiple HNN extension:
$$
G_{n+1} = \langle G_n, \{t_i\}_{i \in \BN} \mid g_i^{t_i} = g_i^2 \rangle
$$
Suppose $x,y \in G_n \setminus \{ 1\}$ are not commensurable. For any $i \in \BN$ and for any $k, l \in \BZ \setminus \{0\}$ it follows that $x^k$ is not conjugate to an element of $\langle g_i \rangle$ or $y^l$ is not conjugate to an element of $\langle g_i \rangle$. By \Cref{hnn-conjugacy-primitivity} it follows that $x$ and $y$ are not commensurable in $G_{n+1}$. Letting $G$ be the direct limit of the $G_n$, we have that there are infinitely many commensurability classes in $G$ and any non-trivial element of $G$ can be written as a proper power of a conjugate of itself. To obtain a finitely generated example, apply \Cref{embedding-conjugacy-theorem}.
\end{proof}

\begin{lem} \label{countable-group-bcyc-commensurability}
There exists a finitely generated torsion-free group $G$ without $b\CYC$ and exactly two commensurability classes.
\end{lem}
\begin{proof}
Let $G_0 = \langle a,b \rangle$ be a free group of rank two. Suppose $G_n$ has been constructed, then enumerate all elements $\{ g_1, g_2, \ldots \}$ of $G_n \setminus \{ 1 \}$ that are not primitive. We form the multiple HNN extension
$$
G_{n+1} = \langle G_n, \{ t_i \}_{i \in \BN} \mid g_i^{t_i} = b^2 \rangle
$$
By \Cref{hnn-extension-primitive-positive-powers} any primitive element of $G_n$ stays primitive in $G_{n+1}$. Similarly, also non-conjugate primitive elements $g,h \in G_n$ will stay non-conjugate in $G_{n+1}$ by \Cref{hnn-conjugacy-primitivity}. If we let $G$ be the direct limit of the $G_n$, then $G$ contains infinitely many primitive conjugacy classes, thus $G$ fails to have $b\CYC$. However, given any non-trivial element $g \in G$, by construction $g^2$ will be conjugate to $b^2$. Thus there are precisely two commensurability classes. Finally, to obtain a finitely generated group with the same properties, apply \Cref{embedding-conjugacy-theorem}.
\end{proof}

\begin{lem}\label{powers-and-centralizers}
Let $G$ be a group such that centralizers of all non-trivial elements are infinite cyclic. Let $a, b \in G$ and suppose that $a^n = b^m$ for some $m,n \neq 0$. If $a$ is primitive, then $b \in \langle a \rangle$.
\end{lem}
\begin{proof}
Note that $C_G(a^n) = \langle x \rangle$ for some $x \in G$ and $a \in C_G(a^n)$ and $b \in C_G(a^n)$, thus $a = x^k$ for some $k$ and $b = x^l$ for some $l$. Since $a$ is primitive it follows that $k = \pm 1$. Hence $b \in \langle a \rangle$.
\end{proof}

In a hyperbolic group the centralizers of infinite order elements are virtually cyclic \cite[III.$\Gamma$.3.10]{BH}. In particular if $G$ is torsion-free hyperbolic and $g$ is a primitive element, then $C_G(\langle g \rangle) = N_G(\langle g \rangle) = \langle g \rangle$.

\begin{lem}\label{HNN-hyperbolic}
Let $G$ be a torsion-free hyperbolic group and let $a,b$ be two primitive elements such that $a$ is not conjugate to $b$. Then the HNN extension $G\ast_{a^t=b^k}$ is hyperbolic for any nonzero interger $k$. 
\end{lem}
\begin{proof}
By \cite[Corollary 1]{KhMy98}, we only need to show that the HNN extension $G\ast_{a^t=b^k}$ is separated. Recall that a subgroup $U$ of $G$ is called conjugate separated if the set $\{ u\in U\mid u^x \in U\}$ is finite for all $x \in G\setminus U$. And an HNN extension $G\ast_{\theta}$ for an isomorphism $\theta \colon U \rightarrow V$ is called separated if either $U$ or $V$ is conjugate separated, and the set $U \cap V^g$ is finite for all $g\in G$.  Now $\langle a\rangle$ is conjugate separated since $C_G(\langle a \rangle) = N_G(\langle a \rangle) = \langle a \rangle$. And if $\langle a\rangle \cap \langle g^{-1}b^kg\rangle$ was non-empty, then $a^n = (g^{-1}bg)^{km}$ for some $m,n\neq 0$. By \Cref{powers-and-centralizers} it follows that $b$ is conjugate to $a$ which contradicts our assumptions.
\end{proof}

\begin{prop} \label{bvc-extension-by-integers}
There exists a countable torsion-free group $G$ and an epimorphism $\varepsilon \colon G \to \BZ$ such that $G$ has $b\CYC$ but $\ker(\varepsilon)$ does not.
\end{prop}
\begin{proof}
Let $G_0 = \langle a,c,d \rangle$, and $\varepsilon_0 \colon G_0 \to \BZ$ be defined by mapping $a$ to $1$ and the other free generators to $0$. Moreover, choose a bijection $\phi_0 \colon \BN_0 \to G_0 \setminus \{ 1 \}$. 

For each $n > 0$ we construct a countable torsion-free group $G_n$, an epimorphism $\varepsilon_n \colon G_n \to \BZ$ and choose a bijection $\phi_n \colon \BN_0 \to G_n \setminus  \{ 1 \}$ such that 

\begin{enumerate}
    \item $G_n$ is either an HNN extension of $G_{n-1}$ through the stable letter $t_n$ or equals $G_{n-1}$. Moreover, $G_n$ is a hyperbolic group.
    \item $\varepsilon_n|_{G_{n-1}} = \varepsilon_{n-1}$ 
    \item Any primitive element $x \in \ker(\varepsilon_{n-1})$ is primitive as an element of $G_n$.
    \item If two primitive elements $x,y \in \ker(\varepsilon_0)$ are conjugate in $G_n$, then the $t_{n-1}$-length of $\omega$ is at most one.
\end{enumerate}

Furthermore we choose the bijection $\phi \colon \BN_0 \to \BN_0 \times \BN_0$ which enumerates the elements of $\BN_0 \times \BN_0$ diagonally, i.e. $\{ (0,0), (1,0), (0,1), (0,2), (1,1), \ldots \}$.

Now, suppose $G_n$ has been constructed. Let $(i,j) = \phi(n)$ and let $g_n = \phi_i(j) \in G_n$. In $G_n$, if $g_n$ is not primitive, or if it is conjugate to an element of $\langle a \rangle$ or $\langle d \rangle$, then set $G_{n+1} = G_n$, $\varepsilon_{n+1} = \varepsilon_n$ and $\phi_{n+1} = \phi_n$. Otherwise we construct $G_{n+1}$ as an HNN extension depending on the value of $\varepsilon_n(g)$ as follows:

\begin{enumerate}[label=(\roman*)]
    \item If $\varepsilon_n(g_n) \neq 0$, we set
        $$
        G_{n+1} = \langle G_n, t_n \mid g_n^{t_n} = a^{\varepsilon_n(g_n)} \rangle.
        $$
    \item If $g_n \in \ker(\varepsilon_n)$, we define
        $$
        G_{n+1} = \langle G_n, t_n \mid g_n^{t_n} = d \rangle.
        $$
\end{enumerate}

Note that in both cases \Cref{HNN-hyperbolic} applies and thus $G_n$ is hyperbolic. 

If $g_n$ is conjugate in $G_n$ to an element of $\langle c,d \rangle$, say $g_n = \alpha_n^{-1} x_n \alpha_n$ for some $x_n \in \langle c,d \rangle$, we define $\varepsilon_{n+1} \colon G_{n+1} \to \BZ$ by $t_n \mapsto |\varepsilon_{n}(\alpha_n)| + 4 \varepsilon(t_{n-1}) + 1$, otherwise we let $t_n \mapsto \varepsilon_n(t_{n-1}) + 1$. Here, we interpret $\varepsilon(t_{-1}) = 1$. Note that $\alpha_n$ and $x_n$ are not unique here, but we fix our choices for each such $g_n$.

Furthermore we choose a bijection $\phi_{n+1} \colon \BN_0 \to G_{n+1}\setminus \{ 1 \}$.

\textbf{Proof of (c).} Let $x \in \ker(\varepsilon_n)$ be primitive.
The claim in case (i) follows directly from \Cref{hnn-conjugacy-primitivity} since $x$ is not conjugate to any element in $\langle g_n \rangle \cup \langle a \rangle$. In case (ii) we can apply \Cref{hnn-extension-primitive} since since $g_n$ is primitive in $G_n$ by assumption and $d \in \ker(\varepsilon_0)$ is primitive in $G_n$ by induction.

\textbf{Proof of (d).} Let $x,y \in \ker(\varepsilon_{0}) \subgroup G_{n+1}$ be primitive and let $\omega \in G_{n+1}$ be a reduced word such that $\omega^{-1} x \omega = y$. If $\omega$ contains no $t_n$ or $t_n^{-1}$, then we are certainly done, thus we can moreover assume that we are in case (ii). Suppose the $t_n$-length of $\omega$ is at least two, then we can write $\omega = \omega_1 t_n^{\pm 1} \omega_2 t_n^{\pm 1} \omega_3$ as a reduced expression, where $\omega_2 \in G_n$ and $\omega_1, \omega_3 \in G_{n+1}$. 

If $\omega = \omega_1 t_n \omega_2 t_n \omega_3$ we know that $t_n^{-1} \omega_1^{-1} x \omega_1 t_n$ has to be a pinch, so $\omega_1^{-1} x \omega_1 \in \langle g_n \rangle$, i.e. $\omega_1^{-1} x \omega_1 = g_n^k$. Since $x$ is primitive as an element of $G_{n+1}$, it follows that $k = \pm 1$. But we also know that $t_n^{-1} \omega_2^{-1} d^k \omega_2 t_n$ has to be a pinch, thus $\omega_2^{-1} d^k \omega_2 = g_n^{m}$ where $m = \pm 1$. Since $\omega_2 \in G_n$, this is impossible by our choice of $g_n$. Similarly the case that $\omega = \omega_1 t_n^{-1} \omega_2 t_n^{-1} \omega_3$ is impossible.

If $\omega$ contains two adjacent stable letters whose exponents are different, we first consider the case that $\omega = \omega_1 t_n \omega_2 t_n^{-1} \omega_3$. Then
$t_n^{-1} \omega_1^{-1} x \omega_1 t_n$ has to be a pinch, thus $\omega_1^{-1} x \omega_1 = g_n^k$ with $k = \pm 1$, so that
$t_n^{-1} g_n^k t_n = d^k$. Now we also know that $t_n \omega_2^{-1} d^k \omega_2 t_n^{-1}$ has to be a pinch, thus $\omega_2^{-1} d^k \omega_2 \in \langle d \rangle$, so $\omega_2 \in \langle d \rangle$ since $G_n$ is hyperbolic. But then $\omega$ was not a reduced word to begin with. Thus we have shown that the reduced element $\omega$ has $t_n$-length at most one. If $\omega = \omega_1 t_n^{-1} \omega_2 t_n \omega_3$ then an analogous argument applies since $\langle g_n \rangle$ is self-normalizing as well since $g_n$ is primitive in case (ii).

We now define $G$ as the direct limit of the $G_n$. Note that the $\varepsilon_n$ induce an epimorphism $\varepsilon \colon G \to \BZ$.

\textbf{G has $b\CYC$.} Let $g \in G$ be a non-trivial element. We claim that it is conjugate to an element in $\langle a \rangle \cup \langle d\rangle $. We can find $n$ such that $g\in G_n$. Since $G_n$ is hyperbolic we can find a primitive element $h$ in $G_n$ such that $g$ is some power of $h$. If $h$ is conjugate to an element of $\langle a \rangle \cup \langle d\rangle$ or if $h$ is conjugate to $g_{n+1}$ in $G_{n+1}$ we are done. Otherwise, $h$ remains primitive in $G_{n+1}$ by \Cref{hnn-conjugacy-primitivity}. On the other hand, our enumeration function $\phi$ and the construction guarantees that in the end any primivite element will be conjugate to an element in $\langle a \rangle \cup \langle d\rangle$. 

\textbf{$\ker(\varepsilon)$ does not have $b\CYC$.}  
We first prove that if $x, y$ are two primitive elements in $\langle c, d \rangle \subgroup \ker(\varepsilon_0)$ and there is some $\omega \in G_{n+1}$ such that $\omega^{-1} x \omega = y$, then $|\varepsilon(\omega)| \leq 2 \varepsilon(t_n)$, where we interpret $\varepsilon(t_{-1}) = 1$ as above. Let $n = -1$, and note that the conjugating element $\omega$ lies in $\langle c,d \rangle$ since $G_0$ is free. In particular, $\varepsilon(\omega) = 0 \leq 2 \varepsilon(t_{-1})$.
Now suppose $n \geq 0$ and let $\omega \in G_{n+1}$ such that $\omega^{-1} x \omega = y$. By (d) it follows that the $t_n$-length of $\omega$ is at most one. If the $t_n$-length equals zero, we are done by induction since $\varepsilon(t_{n-1}) < \varepsilon(t_n)$. If the $t_n$-length is non-zero we can assume without loss of generality that $\omega = \omega_1 t_n \omega_2$ where $\omega_1$ and $\omega_2$ are elements in $G_n$. Since $\omega^{-1} x \omega = y$, it follows that $\omega_1^{-1} x \omega_1 \in \langle g_n \rangle$ and similarly $\omega_2^{-1} d^{\pm 1} \omega_2 = y$. Here, $g_n$ is the element in the kernel of $\varepsilon_n$ that was used to construct $G_{n+1}$ from $G_n$ via an HNN extension. 

Suppose $\omega_1^{-1} x \omega_1 = g_n^{\pm 1}$. Then we also know that $g_n = \alpha_n^{-1} x_n \alpha_n$, thus by induction we know that $|\varepsilon(\omega_1 \alpha_n^{-1})| \leq 2 \varepsilon(t_{n-1})$  since $x^{\omega_1 \alpha_n^{-1}} = x_n^{\pm 1}$, so $|\varepsilon(\omega_1)| \leq |\varepsilon(\alpha_n)| + 2 \varepsilon(t_{n-1})$. Moreover, by induction we also conclude that $|\varepsilon(\omega_2)| \leq 2 \varepsilon(t_{n-1})$. Altogether we obtain
\begin{align*}
|\varepsilon(\omega)| &\leq |\varepsilon(\omega_1)| + \varepsilon(t_n) + |\varepsilon(\omega_2)| \\
& \leq 2 \varepsilon(t_n) \, .
\end{align*}

We are now ready to show that $\ker(\varepsilon)$ contains infinitely many primitive conjugacy classes. Let $x,y$ be two primitive elements in subgroup $\langle c,d \rangle \subgroup \ker(\varepsilon_0)$ such that $x$ is not conjugate to $y$ in $\ker(\varepsilon_0)$. Suppose $x$ and $y$ are conjugate via some $\omega \in \ker(\varepsilon)$, i.e. $\omega^{-1} x \omega = y$. There is some $n \in \BN$ such that $x,y,\omega \in G_{n+1}$. As in the proof above we can assume without loss of generality that $\omega = \omega_1 t_n \omega_2$ with $\omega_1, \omega_2 \in G_n$ and $\omega_1^{-1} x \omega_1 = g_n^{\pm 1}$, $\omega_2 d^{\pm 1} \omega_2 = y$ such that $|\varepsilon(\omega_2)| \leq 2 \varepsilon(t_{n-1})$ and $|\varepsilon(\omega_1)| \leq |\varepsilon(\alpha_n)| + 2 \varepsilon(t_{n-1})$. But since we chose $\varepsilon(t_n) > |\varepsilon_{n}(\alpha_n)| + 4 \varepsilon(t_{n-1})$, we see that $\omega$ cannot lie in the kernel of $\varepsilon$.

\end{proof}

\begin{thm} \label{bvc-fg-extension-by-integers}
There exists a finitely generated torsion-free group $G = H \rtimes \BZ$ such that $G$ has $b\CYC$ but $H$ does not. 
\end{thm}
\begin{proof}
Let $C$ be a torsion-free countable group having $b\CYC$ and admitting an epimorphism $\varepsilon \colon C \to \BZ$ such that $\ker(\varepsilon)$ contains infinitely many conjugacy classes that are primitive in $C$, see \Cref{bvc-extension-by-integers}. 
Let
$$
G(0) = C * F(x,y).
$$
We extend $\varepsilon$ to $\alpha_0 \colon G(0) \to \BZ$ by mapping $x$ and $y$ to $1 \in \BZ$. Moreover, we enumerate the elements of $C = \{ c_0 = 1, c_1, c_2, \ldots \}$ and those of $G(0) = \{ g_0 = 1, g_1, g_2, \ldots \}$. Essentially the same argument as in the proof of \cite[Theorem 7.2]{HuOs} applies. There is a sequence of groups $G(i)$ and epimorphisms $\alpha_i \colon G(i) \to \BZ$ such that $G(i+1)$ is a quotient of $G(i)$ and $\alpha_i$ descends to an epimorphism $\alpha_{i+1} \colon G(i+1) \to \BZ$. Under the quotient maps the group $C$ embeds into $G(i)$ such that $G(i)$ is hyperbolic relative to $C$. Moreover, in $G(i)$ : (a) the elements $c_1, \ldots, c_i$ are contained in $\langle x, y \rangle$ and (b) the elements $g_1, \ldots, g_i$ are conjugate to elements of $C$. If we define $G$ as the direct limit of the $G(i)$, we obtain an induced epimorphism $\alpha \colon G \to \BZ$. By (a) $G$ is 2-generated and by (b) $G$ has $b\CYC$ since $C$ has $b\CYC$.
Since $G(i)$ is hyperbolic relative to $C$, if elements $c, c' \in C$ are conjugate in $G(i)$, i.e. $w c w^{-1} = c'$ for some $w \in G(i)$, then $w \in C$ by \Cref{malnormal-parabolics}. Thus \Cref{rel-hyperbolic-primitive} together with the previous observation imply that $\ker(\alpha)$ contains infinitely many primitive conjugacy classes.
\end{proof}

\begin{prop}
There exists a countable torsion-free group $G$ and an epimorphism $\varepsilon \colon G \to \BZ$ such that both $G$ and $\ker(\varepsilon)$ have $b\CYC$, and $\ker(\varepsilon)$ contains a non-abelian free subgroup.
\end{prop}
\begin{proof} We will first inductively construct a specific sequence $G_0 \subgroup G_1 \subgroup \ldots$ of groups, together with epimorphisms $\varepsilon_n \colon G_n \to \BZ$ such that $\varepsilon_{n+1}|_{G_n} = \varepsilon_{n}$ as follows: We let $G_0 = \langle a,b,c \rangle$ be a non-abelian free group and let $\varepsilon_0 \colon G_0 \to \BZ$ be defined by mapping $a$ to $1$ and $b, c$ to $0$. 
Now, suppose $G_n$ has been constructed such that $G_{n-1} \subgroup G_n$, together with an epimorphism $\varepsilon_n \colon G_n \to \BZ$. We enumerate all elements of $G_n = \{ 1 = g_0, g_1, g_2, \ldots \}$ and define $G_{n+1}$ as the following multiple HNN extension
$$
G_{n+1} = \langle G_n, \{t_i\}_{i \in \BN} \mid g_i^{t_i} = a^{\varepsilon_n(g_i)} \text{ if } \varepsilon_n(g_i) \neq 0 \text{ and } g_i^{t_i} = b \text{ otherwise } \rangle.
$$
We can then extend $\varepsilon_n$ to $G_{n+1}$ to define $\varepsilon_{n+1}$, by mapping the stable letters $t_i$ to $0 \in \BZ$. Finally, we let $G$ be the direct limit of the $G_n$. Observe that $G$ has $b\CYC$ with witnesses $\langle a \rangle$ and $\langle b \rangle$ and that there is an induced epimorphism $\varepsilon \colon G \to \BZ$ such that $\langle b,c \rangle \subgroup \ker(\varepsilon)$. Moreover, since the stable letters of the HNN extensions are contained in $\ker(\varepsilon)$, also $\ker(\varepsilon)$ has $b\CYC$ with only one cyclic group $\langle b\rangle $ as the witness.
\end{proof}

Using arguments as in the proof of \cite[Theorem 7.2]{HuOs} one can construct a group as in the previous proposition that is additionally finitely generated.

\begin{prop}\label{bvc-exp-conj-growth}
There exists a torsion-free finitely generated group $G$ of exponential conjugacy growth that has $b\CYC$ .
\end{prop}
\begin{proof}
As a first step, we will prove that there exists a torsion-free countable group $G$ that contains an infinite cyclic subgroup $\langle a \rangle$ such that $G$ has $b\CYC$ and the elements $a^{2k+1}$ are pairwise non-conjugate. Moreover, there is an element $t \in G$, such that $a^t = a^2$. In the end, the latter property will ensure that the word length of an element $a^m$ is of the order of $\log(|m|)$. We will construct $G$ as direct limit of countable groups $G_0 \subgroup G_1 \subgroup \ldots$ as follows: We let $G_0 = \langle a,t \mid a^t = a^2 \rangle$ be the Baumslag-Solitar group $BS(1,2)$. Note that this group has exponential conjugacy growth, as the elements $a^{2k+1}$ are pairwise non-conjugate, see \cite[Example 2.3]{GubaSapir2010}. To construct $G_{n+1}$ from $G_n$ inductively, we first enumerate all elements in $G_n \setminus (\bigcup_{g\in G} \langle a^g \rangle) = \{g_1, g_2, \ldots \}$ and then form the multiple HNN extension
$$
G_{n+1} = \langle G_n, \{s_i\}_{i \in \BN} \mid g_i^{s_i} = t \rangle. 
$$
It follows inductively from \Cref{hnn-conjugacy-primitivity} that the elements $a^{2k+1}$ are pairwise non-conjugate viewed as elements of $G_{n+1}$. The direct limit $G = \bigcup_{n \geq 0} G_n$ then satisfies the previously required properties as any element in $G \setminus (\bigcup_{g\in G} \langle a^g \rangle)$ will be conjugate to the element $t \in G_0$.

In the second step we want to construct a group with the same properties but which is additionally finitely generated. In the following we will write $C$ instead of $G$ for the previously constructed countable group. We let 
$$
G(0) = C * F(x,y)
$$
and enumerate the elements of $C = \{ c_0 = 1, c_1, c_2, \ldots \}$ resp. $G(0) = \{ g_0 = 1, g_1, g_2, \ldots \}$. Note that $G(0)$ is hyperbolic relative to $C$ and $x, y$ generate a suitable subgroup of $G(0)$. In the following we will not distinguish notationally between elements of $G(0)$ and their representatives in quotients of $G(0)$. We will inductively construct quotient groups $G(i)$ of $G(0)$ such that
\begin{enumerate}
    \item the subgroup $C$ embeds under the quotient map into $G(i)$. Again, we will not distinguish between $C$ and its image in $G(i)$.
    \item $G(i)$ is torsion-free and  hyperbolic relative to $C$. Moreover, $x$ and $y$ generate a suitable subgroup of $G(i)$.
    \item The elements $c_j$ for $1 \leq j \leq i$, considered as elements in $G(i)$, lie in the subgroup generated by $x$ and $y$.
    \item In $G(i)$, for $1 \leq j \leq i$, the elements $g_j$ are conjugate to elements in $C$.
    \item The elements $a^{2k+1}$ are pairwise non-conjugate in $G(i)$.
\end{enumerate}

To construct $G(i+1)$ from $G(i)$, we proceed as follows: If $g_{i+1}$ is parabolic, we set $G'(i) = G(i)$, otherwise we choose an isomorphism $\iota \colon E_{G(i)}(g_{i+1}) \to \langle t \rangle$ and form the HNN extension
$$
G'(i) = \langle G(i), s \mid e^s = \iota(e), e \in E_{G(i)}(g_{i+1}) \rangle.
$$
Note that $\langle x,y \rangle$ is still a suitable subgroup of $G'(i)$  \cite[Corollary 2.16]{HuOs}.

In the next step, we apply \cite[Theorem 6.2]{HuOs} to $G'(i)$ and the elements $\{s, c_{i+1}\}$ resp. $\{ c_{i+1} \}$ in the case that $g_{i+1}$ is parabolic, to obtain a quotient $G(i+1)$ of $G'(i)$. Since $s$ will lie in the subgroup $\langle x,y\rangle$, we obtain a quotient map $G(i) \to G(i+1)$. By construction, (d) holds for $G(i+1)$. The properties (a)-(c) follow directly from \cite[Theorem 6.2]{HuOs}. The last statement (e) is a consequence of \Cref{malnormal-parabolics} and the fact that $G(i+1)$ is torsion-free.

Finally, let $G$ be defined as the direct limit of the $G(i)$. By (c) it follows that $G$ is 2-generated, property (d) implies that $G$ has $b\CYC$ since the same was already true for $C$. Moreover, by (e), the elements $a^{2k+1} \in C$ are pairwise non-conjugate in $G$ as well. Thus $G$ has exponential conjugacy growth.
\end{proof}

\begin{prop}
Let $Q$ be a countable group with $n$ conjugacy classes. Then there exists a torsion-free countable group $G$ with $n+1$ conjugacy classes such that $Q$ is a quotient of $G$.
\end{prop}
\begin{proof}
Let $\{ q_0 = 1, q_1, \ldots, q_{n-1} \} \subset Q$ be representatives of conjugacy classes of elements in $Q$. We let $G_0$ be a countable free group and $\varepsilon_0 \colon G_0 \to Q$ be an epimorphism. We can moreover assume that $\ker(\varepsilon_0)$ is non-trivial. Choose preimages $g_1, \ldots, g_{n-1}$ of $q_1, \ldots, q_{n-1}$ under $\varepsilon_0$ and choose some non-trivial element $g_0 \in \ker(\varepsilon_0)$. 

We now inductively define countable torsion-free groups $G_m$ together with epimorphisms $\varepsilon_m \colon G_m \to Q$ such that $G_{m-1} \subgroup G_m$ and $\varepsilon_{m}|_{G_{m-1}} = \varepsilon_{m-1}$.

Suppose $G_m$ and $\varepsilon_m$ have already been constructed. Enumerate all non-trivial elements $\{ h_1, h_2, \ldots \}$ of $G_m$ and form the multiple HNN extension
$$
G_{m+1} = \langle G_m, \{ t_i \}_{i \in \BN} \mid \text{ relations explained below } \rangle.
$$
We know that $\varepsilon_n(h_i)$ is conjugate to $q_{j_i}$ for some $j_i$. If $q_{j_i} = 1$, then we impose the relation $h_i^{t_i} = g_0$.
Otherwise there is some $\alpha_i \in G_n$ such that $\varepsilon_m( \alpha_i^{-1} h_i \alpha_i) = q_{j_i}$. We then impose the relation $(\alpha_i^{-1} h_i \alpha_i)^{t_i} = g_i$. With these choices we can extend $\varepsilon_m$ to $\varepsilon_{m+1} \colon G_{m+1} \to Q$ by mapping all stable letters $t_i$ to the trivial element in $Q$.

Finally, we let $G$ be the direct limit of the groups $G_m$. By construction, representatives of the conjugacy classes of $G$ are $\{1, g_0, g_1, \ldots, g_{n-1} \}$.
\end{proof}

Note that the above construction is optimal in the sense that if $Q$ has $n$ conjugacy classes and contains torsion then any torsion-free group $G$ that surjects onto $Q$ must have at least $n+1$ conjugacy classes.

Now let $Q$ be a countable group with finitely many conjugacy classes but infinitely many conjugacy classes of finite subgroups. Then $Q$ has $b\CYC$ but not BVC, see \Cref{ex-bcyc-but-not-bvc} for the construction of such a group $Q$. Thus the previous proposition also shows:
 
\begin{cor}\label{bvc-no-inheritance-to-quotients}
There exists a group $G$ with BVC such that a quotient of $G$ fails to have BVC.
\end{cor}

\printbibliography

\end{document}